\documentclass[11pt,a4paper]{article}
\usepackage{amsmath}
\usepackage{amssymb}
\usepackage{amsthm}
\usepackage{latexsym}
\usepackage{pstricks}
\usepackage{amsbsy}
\usepackage{amscd}
\usepackage{mathrsfs}
\usepackage{tipa}
\usepackage{txfonts}
\usepackage{amsfonts}
\usepackage{graphicx}
\usepackage{listings}
\newtheorem{thm}{Theorem}[section]
\newtheorem{cor}[thm]{Corollary}
\newtheorem{lem}[thm]{Lemma}

\newtheorem{exm}{Example}
\newtheorem{prop}[thm]{Proposition}
\newtheorem{defn}[thm]{Definition}
\newtheorem{rem}[thm]{Remark}

\begin{document}

\begin{center}
{\Large \bf Cotorsion pairs and t-structures in a $2-$Calabi-Yau triangulated category\footnote{Supported by the NSF of China (Grants 11131001)}}

\bigskip

{\large Yu Zhou and
 Bin Zhu}
\bigskip

{\small
\begin{tabular}{cc}
Department of Mathematical Sciences & Department of Mathematical
Sciences
\\
Tsinghua University & Tsinghua University
\\
  100084 Beijing, P. R. China &   100084 Beijing, P. R. China
\\
{\footnotesize E-mail: yu-zhou06@mails.tsinghua.edu.cn} &
{\footnotesize E-mail: bzhu@math.tsinghua.edu.cn}
\end{tabular}
}
\bigskip

\today

\end{center}

\begin{abstract}
For a Calabi-Yau triangulated category $\mathcal{C}$ of Calabi-Yau dimension $d$ with a $d-$cluster tilting subcategory $\mathcal{T}$,
it is proved that the decomposition of $\mathcal{C}$ is determined by the special
decomposition of $\mathcal{T}$, namely, $\mathcal{C}=\oplus_{i\in I}\mathcal{C}_i$, where $\mathcal{C}_i, i\in I$ are triangulated subcategories,
 if and only if $\mathcal{T}=\oplus_{i\in I}\mathcal{T}_i,$ where $\mathcal{T}_i, i\in I$
are subcategories with $\mbox{Hom} _{\mathcal{C}}(\mathcal{T} _i[t],\mathcal{T} _j)=0, \forall 1\leq t\leq d-2$ and $i\not= j.$
 This induces that the Gabriel quivers of endomorphism algebras of any two cluster tilting objects in a $2-$Calabi-Yau triangulated category are connected or not at the same time. As an application, we prove that indecomposable $2-$Calabi-Yau triangulated categories with cluster tilting objects have no non-trivial t-structures and no non-trivial co-t-structures. This allows us to give a classification of cotorsion pairs in this triangulated category. Moreover the hearts of cotorsion pairs in the sense of Nakaoka are equivalent to the module categories over the endomorphism algebras of the cores of the cotorsion pairs.

\end{abstract}

\def\s{\stackrel}
\def\Longrightarrow{{\longrightarrow}}
\def\A{\mathcal{A}}
\def\B{\mathcal{B}}
\def\C{\mathcal{C}}
\def\D{\mathcal{D}}
\def\T{\mathcal{T}}
\def\R{\mathcal{R}}
\def\P{\mathcal{P}}
\def\S{\mathcal{S}}
\def\H{\mathcal{H}}
\def\U{\mathscr{U}}
\def\V{\mathscr{V}}
\def\M{\mathscr{M}}
\def\N{\mathcal{N}}
\def\W{\mathscr{W}}
\def\X{\mathscr{X}}
\def\Y{\mathscr{Y}}
\def\Z{\mathcal {Z}}
\def\I{\mathcal {I}}
\def\add{\mbox{add}}
\def\Aut{\mbox{Aut}}
\def\coker{\mbox{coker}}
\def\deg{\mbox{deg}}
\def\diag{\mbox{diag}}
\def\dim{\mbox{dim}}
\def\End{\mbox{End}}
\def\Ext{\mbox{Ext}}
\def\Hom{\mbox{Hom}}
\def\Gr{\mbox{Gr}}
\def\id{\mbox{id}}
\def\Im{\mbox{Im}}
\def\ind{\mbox{ind}}
\def\mod{\mbox{mod}}
\def\mul{\multiput}
\def\c{\circ}
\def \text{\mbox}

\hyphenation{ap-pro-xi-ma-tion}

\textbf{Key words.} Calabi-Yau triangulated category; $d-$cluster tilting subcategory; Cotorsion pair; t-structure; Mutation of cotorsion pair, Heart.
\medskip

\textbf{Mathematics Subject Classification.} 16E99; 16D90; 18E30

\section{Introduction}

Cotorsion pairs (equivalently, torsion pairs) give a way to construct the whole categories from certain special subcategories. They are important in the study of triangulated categories and abelian categories. We recall the definition here. Let $\X, \Y$ be (additive) subcategories in a triangulated category $\C$ with shift functor $[1]$. The pair $(\X,\Y)$ is called a torsion pair in $\C$ provided the following conditions are satisfied:
\begin{enumerate}
  \item $\Hom(X,Y)=0$ for any $X\in \X$, $Y\in \Y$; and
  \item for any $C\in \C$, there is a triangle $X\rightarrow C\rightarrow Y\rightarrow X[1]$ with $X\in \X, Y\in \Y$.
\end{enumerate}
This notion was introduced by Iyama-Yoshino \cite{IY08}, see also \cite{KR07}, which is the triangulated version of the notion with the same name in abelian categories introduced by Dickson \cite{D66} (see the introduction to \cite{ASS06} for further details). The notion of torsion pairs unifies the notion of t-structures in the sense of \cite{BBD81}, co-t-structures in the sense of Pauksztello \cite{P} and \cite{Bon10}, and the notion of cluster tilting subcategories (objects) in the sense of Keller-Reiten \cite{KR07}, see also \cite{BMRRT06}.

Torsion pairs are important in the study of the algebraic structure and geometric structure of triangulated categories.
 Iyama and Yoshino \cite{IY08} use them to study the mutation of cluster tilting subcategories in
triangulated categories, see also \cite{KR07,BRe07}. Nakaoka \cite{Na11} use them to unify the constructions of abelian categories appearing as quotients of triangulated categories by cluster tilting subcategories \cite{BMR07,KR07,KZ}, and the construction of abelian categories as hearts of t-structures \cite{BBD81}.
There is a relation between t-structures and stability conditions in triangulated categories, see \cite{Bri} for details. As one of important special cases, cluster tilting objects (or subcategories)
appeared naturally in the study on the categorification of cluster
algebras \cite{BMRRT06}. They have many nice algebraic properties and
combinatorial properties which have been used in the
categorification of cluster algebras (see the surveys \cite{K12,Re10} and the references therein). In this categorification, the
cluster tilting objects in the cluster category of an acyclic quiver
(or more general a quiver with potential) corresponds to the
clusters of the corresponding cluster algebra.

Cluster tilting subcategories in triangulated categories are the
torsion classes of some special torsion pairs. A triangulated
category (even a $2-$Calabi-Yau triangulated category) may not admit
any cluster tilting subcategories \cite{KZ,BIKR}. In contrast, they always
admit torsion pairs, for example, the trivial torsion pair: (the
whole category, the zero category). In a triangulated
category $\C$ with shift functor $[1]$, when $(\X,\Y)$ is a torsion pair, we call the pair $(\X,\Y[-1])$ is a cotorsion pair, and call the subcategory $\X\bigcap \Y[-1]$ the core of this cotorsion pair.  It follows that $(\X,\Y)$ is a cotorsion pair in $\C$ if and only if $(\X, \Y[1])$ is a torsion pair.

  Recently there are several works on the classification of torsion pairs (or equivalently, cotorsion pairs) of a $2-$Calabi-Yau triangulated category.  Ng gives a
classification of torsion pairs in the cluster categories of
$A_{\infty}$ \cite{Ng10} by defining Ptolemy diagrams of an $\infty-$gon
$P_{\infty}$. Holm-J{\o}rgensen-Rubery \cite{HJR1} gives a classification of cotorsion pairs in cluster category $\C_{A_n}$ of type $A_n$ via Ptolemy diagrams of a regular $(n+3)-$gon $P_{n+3}$.
They also do the same thing for cluster tubes \cite{HJR2}. In \cite{ZZ2}, we define the mutation of
torsion pairs to produce new torsion pairs by generalizing the mutation of cluster tilting subcategories \cite{IY08},
and show that the mutation of torsion pairs has the geometric meaning when the categories have geometric models.
In \cite{ZZZ}, together with zhang, we give the classification of (co)torsion pairs in the (generalized) cluster categories associated with marked Riemann surfaces without punctures. For classification of torsion pairs in  an abelian category, we refer to the recent work of Baur-Buan-Marsh \cite{BBM11}.

In this paper, we show
that an indecomposable $2-$Calabi-Yau triangulated category $\C$ with a cluster tilting
object has only trivial t-structures, i.e. $(\C,0)$, or $(0,\C)$.
For this, we prove the fact that the decomposition of $\C$ is determined by the decomposition of the cluster tilting subcategory.
This decomposition result holds for arbitrary $d-$Calabi-Yau triangulated categories, where $d>1$ is an integer.
As an application of the result on t-structures, we give a classification of cotorsion pairs in $\C$ and determine the hearts of cotorsion pairs in the sense of Nakaoka \cite{Na11}   , which are equivalent to the module categories of their cores. We also discuss the relation between mutation of cotorsion pairs \cite{ZZ2} with mutation of cluster tilting objects.

This paper is organized as follows: In Section 2, some basic
  definitions and results on cotorsion pairs are recalled.
 In Section 3, the definition of decomposition of triangulated categories is recalled.
The decomposition of $d-$cluster tilting categories is defined, which is not only the
decomposition of additive categories, but also with some additional vanish condition on negative extension groups (appeared first in Section 4.2, in \cite{KR08}; and for $d=2$, this condition is empty).
 An example is given to explain in general the decomposition of triangulated categories is not determined by
that of cluster tilting subcategories. It is proved that for any $d-$Calabi-Yau triangulated category, its decomposition is determined by the decomposition of a $d-$cluster tilting subcategory.
   In Section 4, the first main result is that the indecomposable $2-$Calabi-Yau triangulated categories with cluster tilting objects have no non-trivial
   t-structures (Theorem 4.1).  This allows us to give a classification of cotorsion
   pairs in these categories (Theorem 4.4), which is the second main result in this section.
In Section 5, we discuss the relation between mutation of cotorsion pairs and mutation of cluster tilting objects.
For any cotorsion pair $(\X, \Y)$ with core $I$, any basic cluster tilting object $T$ containing $I$ as a direct summand can be written uniquely as $T=T_{\X}\oplus I\oplus T_{\Y}$
such that $T_{\X}\oplus I$
(or $T_{\Y}\oplus I$) is cluster tilting in $\X$ ($\Y$ respectively), which we shall define in this section, and any  triple $(M, I, N)$ of objects $M, I, N$ in $\C$ with the property above
gives a cluster tilting object $M\oplus I\oplus N$ containing $I$ as a direct summand in $\C$. The mutation of such
$T$ in the indecomposable object $T_0$ can be made inside $T_{\X}\oplus I$ or $T_{\Y}\oplus I$, depending on that $T_0$ is a direct summand of $T_{\X}$ or $T_{\Y}$ respectively,
if $T_0$ is not the direct summand of $I$. If $T_0$ is the
direct summand of $I$, then the mutation $T'$ of $T$ in $T_0$ is the cluster tilting object which can be written as $T'=T'_{\X'}\oplus I'\oplus T'_{\Y'}$,
where $(\X',\Y')$ is the mutation of $(\X,\Y)$ and $I'$ is the core
 of $(\X',\Y')$. In the final section, we prove that for any cotorsion pair $(\X,\Y)$ with core $I$ in a $2-$Calabi-Yau triangulated category with a cluster tilting object, the heart $\underline{H}$ of $(\X, \Y)$ is equivalent to mod $I$, where $\underline{H}$ is the subcategory of $\C/I$ which is the image of the subcategory $(\X[-1]*I)\bigcap (I*\Y[1])$ under the natural projection. $\underline{H}$ is called the heart of the cotorsion pair $(\X,\Y)$ \cite{Na11}.

\medskip

\section{Preliminaries}

Throughout this paper, $k$ denotes a field. When we say that $\C$ is a triangulated category, we always assume that $\C$ is a Hom-finite Krull-Schmidt $k-$linear triangulated category over a fixed field $k$.
Denote by [1] the shift functor in $\C$, and by [-1] the inverse of [1]. For a subcategory $\D$, we mean $\D$ is a full subcategory of $\C$ which is closed under
isomorphisms, finite direct sums and direct summands. In this sense, $\D$ is determined by the set of indecomposable objects in it. By $X\in \C$, we mean that $X$ is an object
of $\C$. We denote by add$X$ the additive closure generated by object $X$, which is a subcategory of $\C$.
 Sometimes, we identify an object with the set of indecomposable objects
appearing in its direct sum decomposition, and with the subcategory add$X$. Moreover, if a subcategory $\D$ is closed under [1], [-1] and extensions, then
$\D$ is a triangulated subcategory of $\C$ (in fact it is a thick subcategory).
We call that a triangulated category $\C$ has Serre functor provided there is an equivalent functor $S$ such that $\Hom_{\C}(X, Y)\cong D\Hom _{\C}(Y, SX)$,
which are functorially in both variables, where $D=\Hom_k(-,k)$. If the Serre functor is $[d]$, an integer, $\C$ is called a $d-$Calabi-Yau ($d-$CY, for short) triangulated category. We always use $\Hom(X, Y)$ to denote Hom-space of objects $X, Y$ in $\C$.
We denote by $\Ext^n(X,Y)$ the space
$\Hom(X,Y[n])$.

For a subcategory $\X$ of $\C$, denoted by $\X\subset\C$, let
$$\X ^{\bot}=\{Y\in\C\mid\Hom(X,Y)=0 \mbox{~for any~} X\in \X\}$$
and
$$^{\bot}\X=\{Y\in\C\mid\Hom(Y,X)=0 \mbox{~for any~} X\in \X\}.$$

For two subcategories $\X,
\Y$, by $\Hom(\X, \Y)=0$, we mean that $\Hom(X,Y)=0$
for any $X\in \X$ and any $Y\in \Y$.  A subcategory $\X$ of $\C$ is said to be a rigid subcategory if $\Ext^1(\X,\X)=0$. Let
$$\X*\Y=\{Z\in\C\mid \exists ~~\text{ a triangle }  X\rightarrow Z\rightarrow Y\rightarrow X[1] \mbox{~in $\C$~with~} X\in \X, Y\in \Y\}.$$
It is  easy to see that $\X*\Y$ is
closed under taking isomorphisms and finite direct sums. A subcategory $\X$ is said to be closed under extensions (or an extension-closed
 subcategory) if $\X*\X\subset \X$. Note that $\X*\Y$ is closed under taking direct
summands if $\Hom(\X,\Y)=0$ (Proposition 2.1(1) in \cite{IY08}). Therefore, $\X*\Y$ can be understood as a subcategory of $\C$ in this case.

We recall the definition of cotorsion pairs in a triangulated category $\C$ from \cite{IY08,Na11}.

\begin{defn} Let $\X$ and $\Y$ be subcategories of a triangulated category $\C$.
\begin{itemize}
\item[$1.$] The pair $(\X,\Y)$ is a cotorsion pair if
$$\Ext^1(\X,\Y)=0\text{ and }\C=\X*\Y[1]\text{.}$$
Moreover, we call the subcategory $\I=\X\bigcap \Y $ the core of the cotorsion pair $(\X,\Y)$.
\item[$2.$] A t-structure $(\X, \Y )$ in $\C $ is a cotorsion pair such that $\X$ is closed under $[1]$ (equivalently
$\Y $ is closed under $[-1]$). In this case $\X\bigcap \Y[2]$ is an
abelian category, which is called the heart of $(\X,\Y)$ \cite{BBD81,BRe07}.

\item[$3.$] A co-t-structure $(\X, \Y )$ in $\C $ is a cotorsion pair such that $\X$ is closed under $[-1]$ (equivalently
$\Y $ is closed under $[1]$) \cite{Bon10,P}.

\item[$4.$] The subcategory  $\X$ is said to be a cluster tilting
 subcategory if $(\X, \X)$ is a cotorsion pair \cite{KR07,KZ,IY08}. We say that an object $T$ is a cluster tilting object
 if $\add T$ is a cluster tilting subcategory.
\end{itemize}
\end{defn}

 \begin{rem} A pair
$(\X,\Y)$ of subcategories of $\C$ is called
 a torsion pair if $\Hom(\X,\Y)=0$ and $\C=\X*\Y$. In this case,
 $\I=\X \bigcap\Y[-1]$ is called the core of the torsion pair. Moreover,
 a pair $(\X,\Y)$ is a cotorsion pair if and only if $(\X, \Y [1])$
is a torsion pair. In any case, the core $\I$ is a rigid subcategory of $\C$.
\end{rem}

 \begin{rem} $(\C,0)$ and $(0,\C)$ are t-structures in $\C$, which are called trivial t-structures. They are also co-t-structures  and  are called trivial co-t-structures in $\C$.
\end{rem}

\begin{lem} \cite{ZZ1} Let $(\X,\Y)$ be a cotorsion pair in $\C$ with core $\I$. Then
\begin{itemize}
\item[$1.$] $(\X,\Y)$ is a t-structure if and only if $\I=0$

\item[$2.$] $\X$ is a rigid subcategory if and only if $\X=\I$

\item[$3.$] $\X$ is a cluster tilting subcategory if and only if $\X=\I=\Y$.
\end{itemize}
\end{lem}

\medskip
Recall that a subcategory $\X$ is said to be contravariantly finite in $\C$,
 if any object $M\in \C$ admits a right
$\X-$approximation $f:X\rightarrow M$, which means that any map from
$X'\in \X$ to $M$ factors through $f$. The left $\X-$approximation
of $M$ and covariantly finiteness of $\X$ can be defined dually.
$\X$ is called functorially finite in $\C$ if $\X$ is both
covariantly finite and contravariantly finite in $\C$.
Note that if $(\X,\Y)$ is a torsion pair, then $\X={}^{\bot}\Y$, $\Y=\X^\bot$, and
it follows that $\X$ (or $\Y$) is a contravariantly (covariantly, respectively) finite and extension-closed subcategory of
$\C$.

\medskip

Let $(\X,\Y)$ be a cotorsion pair with core $\I$ in a triangulated category $\C$. Denote by $H$ the subcategory $(\X*\I[1])\bigcap (\I*\Y[1])$. The image of $H$ under the natural projection $\C\rightarrow \C/\I$, which denoted by $\underline{H}$, is called the heart of the cotorsion pair $(\X,\Y)$. It is proved by Nakaoka that the heart $\underline{H}$ is an abelian category, see \cite{Na11} for more detailed construction.

\section{Decompositions of Calabi-Yau triangulated categories}

In this section, we discuss how the decomposition of triangulated categories is determined by that of a cluster tilting subcategory.
 We recall the definition of $d-$cluster tilting subcategories from \cite{KR07,IY08} in the following:

\begin{defn} Let $\C$ be a triangulated category, $d> 1$, an integer. A subcategory $\T$ of $\C$ is called $d-$rigid provided $\Ext^i(\T,\T)=0$ for all $1\le i\le d-1$.

A $d-$ rigid subcategory $\T$ is called $d-$cluster tilting provided that $\T$ is functorially finite, and satisfies the property: $T\in \T$ if and only if $\Ext^i(\T, T)=0$ for all $1\le i\le d-1$ if and only if
$\Ext^i(T, \T)=0$ for all $1\le i\le d-1;$

An object $T$ is called a $d-$cluster tilting (respectively $d-$rigid) object if $addT$ is $d-$cluster tilting (respectively $d-$rigid).
\end{defn}

The main examples of $d-$cluster tilting subcategories are $d-$cluster tilting subcategories in $d-$cluster categories (see \cite{IY08,T,Zhu}). Other examples can be found in \cite{K08,BIKR}. Note that when $d=2$, the $d-$cluster tilting subcategories (or $d-$cluster tilting objects) are called cluster tilting subcategories (cluster tilting objects respectively).

\begin{defn}\label{decompofC} Let $\C$ be a triangulated category, and $\C_i, i\in I$ be triangulated subcategories of $\C$.
We call that $\C$ is a direct sum of triangulated subcategories $\C_i, i\in I,$ provided that
\begin{itemize}
\item[$1.$] Any object $M\in \C$ is a direct sum of finitely many objects $M_i\in \C_i$;
\item[$2.$] $\Hom(\C_i,\C_j)=0,\forall i\neq j$.
\end{itemize}
In this case, we write $C=\oplus_{i\in I}\C_i$. We say $\C$ is indecomposable if $\C$ cannot be written as a direct sum of two nonzero triangulated subcategories.
\end{defn}

\begin{defn}\label{decompofT} Let $\T$ be a $d-$cluster tilting subcategory of a triangulated category $\C$, and $\T_i, i\in I,$  be subcategories of $\T$.
We call that $\T$ is a direct sum of subcategories $\T_i, i\in I,$ provided that
\begin{itemize}
\item[$1.$] Any object $T\in \T$ is a direct sum of finitely many objects $T_i\in \T_i$;
\item[$2.$] $\Hom(\T_i,\T_j)=0, \forall i\neq j$;
\item[$3.$] $\Hom(\T_i[k],\T_j)=0,\forall i\neq j, 1\leq k\leq d-2$;
\end{itemize}
In this case, we write $\T=\oplus_{i\in I}\T_i$. We say $\T$ is indecomposable if $\T$ cannot be written as a direct sum of two nonzero subcategories.
\end{defn}

\begin{rem} The third condition in Definition \ref{decompofT} appeared first in \cite{KR07} for the study of Gorenstein property of $d-$cluster tilting subcategories (see the subsection 4.6 there for details), and it will play an essential rule in our result.
 When $d=2$, this condition is empty. \end{rem}

The following example shows that there are indecomposable $d-$CY triangulated categories admitting $d-$cluster
tilting subcategories, those cluster tilting subcategories can be decomposed as sum of subcategories satisfying the conditions $1, 2$, but $3$ in Definition 3.3.

\begin{exm} Let $Q:3\rightarrow 2\rightarrow 1$ be the quiver of type $A_3$ with linear orientation, and $\C$ be the $4-$cluster category of $Q$, i.e. $\C=D^b(kQ)/\tau^{-1}[3]$ (compare \cite{K08}).
 Let $P_1, P_2, P_3$ be the indecomposable projective modules associated to the vertices of $Q$, and $S_1, S_2, S_3$ the corresponding simple modules. Then $T=P_1\oplus P_2\oplus P_3$ is
a $4-$cluster tilting object, $P_1\oplus P_3$ is an almost complete
$4-$cluster tilting object, it has $4$ complements (compare \cite{Zhu,T}), one is $P_2$, the others are $S_3, S_3[1],$ and $S_3[2]$. Denote by
$\T =add (P_1\oplus P_3\oplus S_3[1])$, which is a $4-$cluster tilting subcategory of $\C$. Set $\T_1=add(P_1\oplus P_3)$, $\T_2=add S_3[1]$. Both are subcategories of $\T$.
 It is easy to see that $\T, \T_1, \T_2$ satisfy the first two conditions of Definition 3.3, but not satisfy the third one, an easy computation shows $Hom(P_3[1], S_3[1])\not=0$.

We note that this $4-$cluster category $\C$ is indecomposable.

\end{exm}

We will discuss the relation between the decomposition of triangulated categories and the decomposition of $d-$cluster tilting subcategories. Firstly we look at two examples:

\begin{exm} Let $Q$ be a connected quiver without oriented cycles,
$\C=D^b(kQ)$ the bounded derived category of $kQ$. It is an indecomposable triangulated category. We know $\T=add\{ \tau^n[-n] kQ\  |\  n\in \textbf{Z}\ \}$ is a cluster
tilting subcategory containing infinitely many indecomposable objects in $\C$. Let $\T_i=add\{\tau ^i[-i]kQ\}$ for $i\in \textbf{Z}.$ It is easy to check that $\T=\oplus _{i\in \textbf{Z}}\T_i$.

\end{exm}

\begin{exm} Let $Q$ be a connected quiver without oriented cycles, $F=\tau^{-1}[1]$ an automorphism of the derived category $D^b(kQ)$. The repetitive cluster category of $Q$
 is defined for any positive integer $m$, namely, the orbit triangulated category $\C=D^b(kQ)/(F^m)$ \cite{K08}. It is an indecomposable triangulated category. Let $m=2$.
Then $kQ\oplus F(kQ)$ is a cluster tilting object in  $\C$. Let $\T=add(kQ\oplus F(kQ))$, $\T_1=add(kQ)$, $\T_2=add(F(kQ))$. Then $\T$ is a cluster tilting subcategory and $\T=\T_1\oplus \T_2$.

\end{exm}
The two examples above show that in general the indecomposable triangulated category may admit a decomposable $d-$cluster tilting subcategory.
 In the following, we will prove that the
decomposition of $d-$CY triangulated categories is determined by the decomposition of a $d-$cluster tilting subcategory. Recall that a $k-$linear triangulated category $\C$ is $d-$CY if $[d]$ is the Serre functor.

\begin{prop}\label{prop1} Let $\C$ be a $d-$CY triangulated category with a $d-$cluster tilting subcategory $\T$.
Suppose that $\T=\oplus_{i\in I}\T_i$ with $\T_i,i\in I,$ nonzero subcategories,
and let $\C_i=\T_i\ast\T_i[1]\ast\cdots\ast\T_i[d-1]$ for any $i\in I$. Then $\C_i$ is a triangulated subcategory of $\C$ and $\C=\oplus_{i\in I}\C_i$.
\end{prop}

Note that by Proposition 2.1 \cite{IY08}, $\C_i,i\in I,$ are closed under direct summands, so they are subcategories of $\C$.

We divide our proof into several steps:

\begin{lem}\label{lem-decomp}
Under the same assumption as in Proposition \ref{prop1}, every object $X$ in $\C$ has
a decomposition $X=\oplus_{i\in I}X_i$ with finite many nonzero $X_i\in\C_i$, $ i\in I$.
In particular, every indecomposable object of $\C$ lies in some $\C_i, i\in I$.
\end{lem}

\begin{proof}

Since $\T=\oplus_{i\in I}\T_i$ is a $d-$cluster tilting subcategory, by Corollary 3.3 in \cite{IY08},
for each indecomposable object $X$ in $\C$, there are $d$ triangles:
$$X^{(n)}\s{f^{(n)}}\rightarrow \oplus_{i\in J}B_i^{(n-1)}\rightarrow X^{(n-1)}\rightarrow X^{(n)}[1], \text{ } n=1,\cdots,d,$$
where $J$ is a finite subset of $I$, $B_i^{(n-1)}\in\T_i$ $X^{(0)}=X$ and $X^{(d)}=0$.
Then $X^{(d-1)}\cong \oplus_{i\in J}B_i^{(d-1)}$.  We want to prove that $X\cong\oplus_{i\in J}X_i$ with $X_i\in\C_i$.

Assume that $X^{(n)}\cong \oplus_{i\in J}X_i^{(n)}$ with $X_i^{(n)}\in\T_i\ast\T_i[1]\ast\cdots\ast\T_i[d-1-n]$ for some $1\leq n\leq d-1$.
By Definition \ref{decompofT}, $\Hom(\T_i[k],\T_j)=0$ for $i\neq j$, $0\leq k\leq d-1-n\leq d-2$, then $\Hom(X_i^{(n)},\T_j)=0$ for $j\neq i$.
So $f^{(n)}$ is a diagonal map, say
$\begin{pmatrix}f_1&0&0\\0&\ddots&0\\0&0&f_{|J|}\end{pmatrix}$, where $f_i:X_i^{(n)}\rightarrow B_i^{(n-1)}$.
Extend each $f_i$ to triangle:
$$X_i^{(n)}\s{f_i}\rightarrow B_i^{(n-1)}\rightarrow X_i^{(n-1)}\rightarrow X_i^{(n)}[1].$$
Then we have that $X^{(n-1)}\cong \oplus_{i\in J} X_i^{(n-1)}$ and $X_i^{(n-1)}\in\T_i\ast\T_i[1]\ast\cdots\ast\T_i[d-n]$, $i\in J$.
By induction on $n$ (from $d-1$ to $0$), $X=X^{(0)}\cong\oplus_{i\in J}X_i^{(0)}$ with $X_i^{(0)}\in\C_i$.

\end{proof}

\begin{lem}\label{lemofcap}

Under the same assumption as in Proposition \ref{prop1}, $\C_i=\bigcap\limits_{j\neq i}\bigcap\limits_{k=1}^{2d-2}{}^\bot\T_j[k]$
holds for any $i\in I$.
\end{lem}

\begin{proof}

By Definition \ref{decompofT}, $\Hom(\T_i,\T_j[l])=0$ and $\Hom(\T_j,\T_i[l])=0$, for $-(d-2)\leq l\leq d-1$, $i\neq j$.
 Then for $0\leq m\leq d-1$, $1\leq k\leq d-1$, we have that $\Hom(\T_i[m],\T_j[k])\cong\Hom(\T_i,\T_j[k-m])=0$
due to $-(d-2)\leq k-m\leq d-1$, and $\Hom(\T_i[m],\T_j[d+k-1])\cong D\Hom(\T_j[k-1],\T_i[m])=0$ as $-(d-2)\leq m-k+1\leq d-1$.
So $\Hom(\C_i,\T_j[k])=\Hom(\T_i\ast\T_i[1]\cdots\ast\T_i[d-1],\T_j[k])=0$ for $1\leq k\leq 2d-2$, $i\neq j$.
This implies $\C_i\subset\bigcap\limits_{j\neq i}\bigcap\limits_{k=1}^{2d-2}{}^\bot\T_j[k]$.

Fix an element $i\in I$. Let $X$ be an object satisfying $\Hom(X,\T_j[k])=0$ for $1\leq k\leq 2d-2$, $j\neq i$.
By Lemma \ref{lem-decomp}, $X$ has a decomposition $X=\oplus_{l\in J}X_l$, $X_l\in\C_l$, for some finite subset $J$ of $I$.
By the definition of $\C_l$, there are $d$ triangles:
$$X_l^{(n)}\rightarrow A_l^{(n-1)}\s{g_l^{(n-1)}}\rightarrow X_l^{(n-1)}\rightarrow X_l^{(n)}[1], \text{ } n=1,\cdots,d,$$
where $A_l^{(n-1)}\in\T_l$, $X_l^{(0)}=X_l$, $X_l^{(d)}=0$ and $g_l^{(n-1)}$ is the minimal right $\add\T_l-$approximation of $X_l^{(n-1)}$ (compare Corollary 3.3 in \cite{IY08}).
If $l\neq i$, then $g_l^{(0)}=0$ and $B_l^{(0)}=0$ thanks to
$\Hom(\T_l,X)=D\Hom(X,\T_l[d])=0$. So $X_l^{(1)}\cong X_l^{(0)}[-1]=X_l[-1]$.
Assume that $X_l^{(n)}\cong X_l[-n]$ for some $1\leq n\leq d-2$.
Then $g_l^{(n)}=0$ by $\Hom(\T_l,X[-n])\cong D\Hom(X,\T_l[d+n])=0$ and then $X_l^{(n+1)}\cong X_l^{(n)}[-1]\cong X_l[-(n+1)]$.
By induction on $n$, we have that $g_l^{(n-1)}=0$ and $X_l^{(n)}=X_l[-n]$, for $1\leq n\leq d-1$, $l\neq i$.
Note that $X_l^{(d-1)}\cong A_l^{(d-1)}$ by $X_l^{(d)}=0$.
 From that $\Hom(X_l^{(d-1)},X_l^{(d-1)})\cong\Hom(X_l[-(d-1)],A_l^{(d-1)})\cong\Hom(X_l,B_l^{(d-1)}[d-1])=0$, we have  $X_l^{(d-1)}=0$. Then $X_l\cong X_l^{(d-1)}[d-1]=0$ ($l\neq i$). Hence $X\cong X_i\in\C_i$.
\end{proof}

\begin{lem}\label{lemtrisubcat}

Under the same assumption as in Proposition \ref{prop1}, all $\C_i$, $i\in I$, are triangulated subcategories of $\C$.
\end{lem}

\begin{proof}

Let $X\rightarrow Z\rightarrow Y\rightarrow X[1]$ be a triangle with $X,Y\in\C_i$. By Lemma \ref{lemofcap},
we have that $\Hom(X,\T_j[k])=0$ and $\Hom(Y,\T_j[k])=0$ and then $\Hom(Z,\T_j[k])=0$ for $1\leq k\leq 2d-2$, $j\neq i$.
By Lemma \ref{lemofcap} again, we have that $Z\in\C_i$. Therefore, $\C_i$ is closed under extensions.

For any $i$, $\T, \T[1], \cdots, \T[d-1]$ are included in $\C$. We claim that $\T_i[d]$ is a subcategory of $\C_i$.
Otherwise, there is an indecomposable object of $\T_i$, say $X$,
such that $X[d]$ is not an object of $\C_i$. Then by Lemma \ref{lem-decomp}, $X[d]$ is in $\C_j$ for some $j\neq i$.
Note that $\Hom(X,X[d])\cong D\Hom(X,X)\neq0$ which contradicts with $\Hom(\T_i,\C_j)=0$ by Lemma \ref{lemofcap}
and $d-$CY property.
Then we prove that
$\T_i[d]$ is included in $\C_i$. Hence $\C_i[1]=\T_i[1]\ast\cdots\ast\T_i[d]\subset\C_i$, that is, $\C_i$ is closed under [1].
Dually, one can prove that $\C_i$ is closed under [-1]. Therefore, $\C_i$ is a triangulated subcategory of $\C$.
\end{proof}

\textbf{Proof of Proposition \ref{prop1}}.
It is sufficient to verify that $\Hom(\C_i,\C_j)=0$, for $i\neq j$. By Lemma \ref{lemtrisubcat},
$\C_i[-1]=\C_i$, then $\Hom(\C_i,\T_j)=\Hom(\C_i[-1],\T_j)=\Hom(\C_i,\T_j[1])=0$ for $i\neq j$, where the last
equality is due to Lemma \ref{lemofcap}. Then $\Hom(\C_i,\C_j)=\Hom(\C_i,\T_j\ast\T_j[1]\ast\cdots\ast\T_j[d-1])=0$.

\bigskip

The following lemma is a generalization of Remark 2.3 in \cite{ZZ1}.

\begin{lem}\label{lemofremkzz}
Let $\C$ be a triangulated category and $\T$ be a $d-$rigid subcategory of $\C$
satisfying $\C=\T\ast\T[1]\ast\cdots\ast\T[d-1]$. Then $\T$ is a $d-$cluster tilting
subcategory of $\C$.
\end{lem}

\begin{proof}
Note that $(\T,\T[1]\ast\cdots\ast\T[d-1])$ and $(\T\ast\cdots\ast\T[d-2],\T[d-1])$ form two torsion pairs.
So $\T$ is contravariantly finite in $\C$ and $\T[d]$ is covariantly finite in $\C$. Therefore $\T$ is functorially finite in $\C$. Take an object $X$ in $\C$ with $\Hom(X,\T[t])=0$ for $1\leq t\leq d-1$.
Then $\Hom(X,\T[1]\ast\cdots\ast\T[d-1])=0$. Hence $X\in\T$. Similar proof for $X\in \T$ if $\Hom(\T,X[t])=0$ for $1\leq t\leq d-1$.
 Hence $\T$ is $d-$cluster tilting in $\C$.
\end{proof}

Now we prove our main result in this section.

\begin{thm}\label{decomthm}
Let $\C$ be a $d-$CY triangulated category with a $d-$cluster tilting subcategory $\T$.
Then $\C$ is a direct sum of indecomposable triangulated subcategories $\C_i$, $i\in I$ if and only if
 the cluster tilting subcategory $\T$ is a direct sum of indecomposable subcategories $\T_i, i\in I$.
 Moreover $\C_i=\T_i\ast\T_i[1]\ast\cdots\ast\T_i[d-1]$ and $\T_i$ is a $d-$cluster tilting subcategory in $\C_i$, $i\in I$.
\end{thm}

\begin{proof}
We first show the "only if" part. By the definition of direct sums of triangulated subcategories,
any object $T$ in $\T$ has a decomposition $T=\oplus_{i\in J} T_i$ with $J$ a finite subset of $I$, $T_i\in\C_i$
and $\Hom(T_i[k],T_j)=0$ for $0\leq k\leq d-2$, $i\neq j$. Then $\T=\oplus_{i\in I} \T_i$ where $\T_i=\T \bigcap \C_i$.
By Definition \ref{decompofC}, for any object $X\in\C_i$, $\Hom(X,\T_j[k])=0$ for $j\neq i$ and any $k$.
Then by Lemma \ref{lemofcap}, $X\in\T_i\ast\T_i[1]\ast\cdots\ast\T[d-1]$. Hence $\C_i=\T_i\ast\T_i[1]\ast\cdots\ast\T[d-1]$.
By Lemma \ref{lemofremkzz}, $\T_i$ is a $d-$cluster tilting subcategory of $\C_i$.
It follows from that of $\C_i$ and Proposition \ref{prop1} that $\T_i$ is indecomposable.

To prove the "if" part. It follows from Proposition \ref{prop1} that there is a decomposition
$\C=\oplus_{i\in I}\C_i$, where $\C_i=\T_i*\T_i[1]*\cdots*\T_i[d-1]$ is a triangulated subcategory of $\C$.
By Lemma \ref{lemofremkzz}, $\T_i$ is a $d-$cluster tilting subcategory in $\C_i$. If $\C_i$ is not indecomposable, say $\C_i=\C_i'\oplus \C_i''$ with nonzero triangulated subcategories $\C_i',\C_i''$,
 then by the proof of the ``only if'' part, we have
$\T_i=\T_i'\oplus T_i''$, and $\C_i'=\T_i'*\T_i'[1]*\cdots*\T_i'[d-1], \C_i''=\T_i''*\T_i''[1]*\cdots*\T_i''[d-1]$. It follows that $\T_i',\T_i''$ are nonzero subcategories, a contradiction to the indecomposableness of $\T_i$.

The other assertion follows from Lemma \ref{lemofremkzz}.
\end{proof}

We give a simple example for $d=2$.

\begin{exm} Let $Q: 4\rightarrow 3\rightarrow 2\rightarrow 1$, $\mathcal{C} =\mathcal{C} _Q$, the cluster category of $Q$ whose Auslander-Reiten quiver is the following:

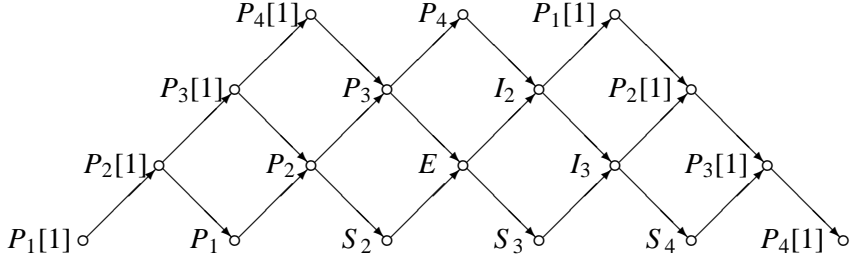
\begin{center}
\setlength{\unitlength}{1cm}
\begin{picture}(10,4)
\mul(0,0)(2,0){6}{$\c$}
\mul(1,1)(2,0){5}{$\c$} \mul(2,2)(2,0){4}{$\c$}
\mul(3,3)(2,0){3}{$\c$}

\mul(0.15,0.15)(2,0){5}{\vector(1,1){0.9}}
\mul(1.15,1.15)(2,0){4}{\vector(1,1){0.9}}
\mul(1.15,1.05)(2,0){5}{\vector(1,-1){0.9}}
\mul(2.15,2.05)(2,0){4}{\vector(1,-1){0.9}}
\mul(2.15,2.15)(2,0){3}{\vector(1,1){0.9}}
\mul(3.15,3.05)(2,0){3}{\vector(1,-1){0.9}}

\put(-0.9,0){${P_1[1]}$} \put(0.1,1){$P_2[1]$}
\put(1.1,2){$P_3[1]$} \put(2.1,3){$P_4[1]$}
\put(1.5,0){$P_1$} \put(2.5,1){$P_2$}
\put(3.5,2){$  P_3$} \put(4.5,3){$  P_4$}
\put(3.5,0){$ S_2$} \put(5.5,0){$ S_3$} \put(4.5,1){$ E$}
\put(5.5,2){$  I_2$} \put(6.5,1){$I_3$} \put(7.5,0){$S_4$}
\put(6,3){${ P_1[1]}$} \put(7,2){$P_2[1]$}
\put(8,1){$P_3[1]$} \put(9,0){$P_4[1]$}
\end{picture}
\end{center}

\bigskip

We take $\mathcal{X}=add(E)$, ${}^{\bot}(\mathcal{X}[1])=add(\{E, P_3, P_4[1],P_4,I_2,P_1[1], S_2, S_3\}$
By \cite{IY08}, the subquotient category ${}^{\bot}(\mathcal{X}[1])/\mathcal{X}=\add(\{ P_3, P_4[1],P_4,I_2,P_1[1], S_2, S_3\})$
 is triangulated, and $2-$CY.
This subquotient category admits cluster tilting objects, for example, the object $T=P_4[1]\oplus P_3\oplus E\oplus S_3$.
 We have that in this subquotient category, $\add T=\add(S_3)\oplus \add(P_3\oplus P_4[1])$. Then by Theorem \ref{decomthm}, this subquotient category
${}^{\bot}(\mathcal{X}[1])/\mathcal{X}=\add(\{S_2,S_3\})\oplus \add(\{ P_3, P_4[1], P_4, I_2, P_1[1]\})$, in which, the first direct summand is equivalent to the cluster category of type $A_1$,
the second one is equivalent to the cluster category of type $A_2$.

\end{exm}

\begin{cor}\label{cor1}
Let $\C$ be a $d-$CY triangulated category admitting a $d-$cluster tilting subcategory $\T$. Then $\C$ is indecomposable if and only if $\T$ is indecomposable.
\end{cor}

\begin{cor}\label{corind}
Let $\C$ be a $d-$CY triangulated category,  $\T$ and $\T'$ be two $d-$cluster tilting subcategories. Then $\T$ is indecomposable if and only if $\T'$ is indecomposable.
\end{cor}

\begin{cor} Let $\C$ be a $d-$CY triangulated category with a $d-$cluster tilting object $T$.
Then $\C$ is a direct sum of finitely many indecomposable triangulated subcategories $\C_i$, $i=1, \cdots m$. Moreover the cluster tilting subcategory $\T=\add T$ is
a direct sum of indecomposable subcategories $\T_i, i=1,\cdots, m$, and $\C_i=\T_i\ast\T_i[1]\ast\cdots\ast\T_i[d-1]$
and $\T_i$ is a $d-$cluster tilting subcategory in $\C_i$, $i=1,\cdots m$.

\end{cor}
\begin{proof}
Any triangulated category can be decomposed as a direct sum of triangulated subcategories. For the $d-$CY triangulated category $\C$ with a $d-$cluster tilting object $T$, the number of direct summands of the decomposition
of $\C$ is finite since that the number of indecomposable direct summands of $T$ is finite. Then we have the decomposition of $\C=\oplus_{i=1}^{m}\C_i$. The other assertion follows directly from Theorem \ref{decomthm}.

\end{proof}

For the special case of $d=2$, i.e., $\C$ is $2-$CY triangulated category with a cluster tilting object $T$, the decomposition of
$\C$ corresponds to the partition of connected components of the Gabriel quiver of $\End(T)$.

\begin{defn}
A basic rigid object $T$ in $\C$ is called connected provided $T$ cannot written as $T=T_1\oplus T_2$ with
property that $T_i\not= 0$, and $Hom(T_i,T_j)=0$, for $i\neq j\in \{1,2 \}$. Any cluster tilting object in $\C$ can be decomposed as a direct sum of connected summands:
 $T=\oplus_{i=1}^{m}T_i$ with $T_i$ being connected. We call such decomposition a complete decomposition of $T$.
\end{defn}

Every $\C$ can be decomposed uniquely to a direct sum of nonzero indecomposable triangulated subcategories.
We call this decomposition is the complete decomposition of $\C$ and denote by $ns(\C)$ the number of indecomposable direct summands of such decomposition of $\C$.
For a cluster tilting object $T$ in $\C$, the Gabriel quiver of $\End(T)$ is denoted by $\Gamma_T$ and the number of connected components of $\Gamma_T$
is denoted by $nc(\Gamma_T)$.
Note that the complete decomposition of $T$ corresponds to the connected components of Gabriel quiver of the $2-$CY tilted algebra $\End (T)$.
So by applying the theorem above, we have the following result immediately.

\begin{cor}\label{cor11}
Let $\C$ be a $2-$CY triangulated category admitting a cluster tilting object $T$. Then the number $nc(\Gamma_T)$ of connected components of the quiver $\Gamma_T$ is equal to $ns(\C)$.
In particular, $\C$ is indecomposable if and only if $\Gamma_T$ is connected.
\end{cor}

\begin{cor}\label{connected}
Let $\C$ be a $2-$CY triangulated category and let $T,T'$ be cluster tilting objects in $\C$. Then $\Gamma_T$ is connected if and only if $\Gamma_{T'}$ is connected.
\end{cor}

\begin{proof}
$\Gamma_T$ is connected $\Leftrightarrow$ $\C$ is indecomposable $\Leftrightarrow$ $\Gamma_{T'}$ is connected.
\end{proof}

\begin{rem}
Let $(S,M)$ be a marked surface and $nc(S)$ denote the number of connected components of $S$. Then $nc(S)=ns(\C(S,M))$ (compare \cite{ZZ2}).
\end{rem}

\section{Classification of Cotorsion pairs in 2-Calabi-Yau categories}

From now on, except Proposition 4.6, we always suppose that the triangulated category $\C$ is $2-$Calabi-Yau ($2-$CY for short), i.e. $[2]$ is the Serre functor of $\C$.


The main examples of $2-$CY triangulated categories are the followings:
\begin{itemize}
\item[$1.$] Cluster categories of hereditary abelian $k-$categories in the sense of \cite{BMRRT06} (also \cite{CCS06} for type $A$); and generalized cluster categories of algebras with global dimension at most $2$ (including
the case of quivers with potentials) in the sense of Amoit \cite{Ami09}. All these $2-$CY triangulated categories have cluster tilting objects.
\item[$2.$] The stable categories of preprojective algebras of Dynkin quivers. They also have cluster tilting objects \cite{GLS,BIRS}.
\item[$3.$] The cluster category of type $A_{\infty}$. It has  cluster tilting subcategories, which contains infinitely many indecomposable objects \cite{KR07,HJ,Ng10}.
\item[$4.$] The bounded derived categories $D^b(mod_{f.l.}\Lambda)$ of modules with finite length over preprojective algebras $\Lambda$ of non-Dynkin quivers. They have no cluster tilting subcategories.
There are many stable
subcategories of mod$_{f.l.}\Lambda $ associated to elements in the Coxeter groups of the quivers. Their stable categories are $2-$CY, and have cluster tilting objects. See \cite{GLS,BIRS} for details.
\item[$5.$] Stable categories of Cohen-Macaulay modules over three-dimensional complete local commutative noetherian Gorenstein isolated singularity containing the residue field \cite{BIKR}.

\end{itemize}

We shall first decide a special kind of cotorsion pairs: t-structures.
Recall that $(\X,\Y)$ is a t-structure in $\C$, if $\Ext^1(\X,\Y)=0$, $\C=\X\ast\Y[1]$ and $\X[1]\subset\X$, $\Y[-1]\subset\Y$.



 The first main result in this section is the following result.

\begin{thm}\label{t-str}
Let $\C$ be an indecomposable 2-CY triangulated category with a cluster tilting object $T$.
Then $\C$ have no non-trivial t-structures, i.e. the t-structures in $\C$ are $(\C,0)$ and $(0,\C)$.
\end{thm}

\begin{proof}
Let $(\X,\Y)$ be a t-structure in $\C$. Put $\T=\add T$. Then for each indecomposable object $T_i\in\T$, $i\in I,$ there is a triangle
$$X_i\s{f_i}\longrightarrow T_i\s{g_i}\longrightarrow Y_i[1]\s{h_i}\longrightarrow X_i[1]$$ with
$X_i\in\X,Y_i\in\Y$. Let $\R$ be the subcategory of $\C$ generated additively by $X_i$, $Y_i$, $i\in I$.
Then $\T\subset\R\ast\R[1]$. We shall prove that $\R$ is a cluster tilting subcategory.

For any map $\alpha\in\Hom(Y_i[1],Y_j[2])$, consider the following diagram:
\begin{center}
\setlength{\unitlength}{1cm}
\begin{picture}(6,2.3)
\put(-0.35,0){$X_j[1]$}\put(1.65,0){$T_j[1]$}
\put(3.65,0){$Y_j[2]$}\put(5.65,0){$X_j[2]$}
\put(-0.15,2){$X_i$}\put(1.85,2){$T_i$}
\put(3.65,2){$Y_i[1]$}\put(5.65,2){$X_i[1]$}
\mul(0.5,0.1)(2,0){3}{\vector(1,0){1}}
\mul(0.5,2.1)(2,0){3}{\vector(1,0){1}}
\put(4,1.8){\vector(0,-1){1.4}}
\mul(3.9,1.9)(2,0){2}{\vector(-1,-1){1.6}}
\put(2.9,2.2){\small$g_i$}
\put(3.75,1){\small$\alpha$}
\put(2.6,0.2){\small$-g_j[1]$}
\put(4.9,2.2){\small$h_i$}
\put(4.6,0.2){\small$-h_j[1]$}
\put(5.3,1){\small$\beta$}
\end{picture}
\end{center}
The composition $-h_j[1]\circ\alpha\in\Hom(Y_i[1],X_j[2])\cong D\Hom(X_j,Y_i[1])=0$,
then $\alpha$ factors through $-g_j[1]$. So $\alpha\circ g_i=0$ due to $\Hom(T_i,T_j[1])=0$.
Therefore $\alpha$ factors through $h_i$, i.e. there is a morphism $\beta\in\Hom(X_i[1],Y_j[2])$ such that
$\alpha=\beta\circ h_i$. But $\Hom(X_i[1],Y_j[2])=0$, so $\alpha=0$.
Then $\Ext^1(Y_i,Y_j)=0$. Dually, we have that $\Ext^1(X_i,X_j)=0$. By the definition of t-structure
and 2-CY property, we also have $\Ext^1(X_i,Y_j)=0$ and $\Ext^1(Y_i,X_j)$. Hence $\R$ is a rigid subcategory.

Given an object $M$ with $\Ext^1(M,X_i)=0$, $\Ext^1(M,Y_i)=0$ for $i\in I$. Since $\T$ is cluster tilting,
there is a triangle $M\s{w}\longrightarrow A\s{u}\longrightarrow B\s{v}\longrightarrow M[1]$ with $A,B\in\T$. Since $\T\subset\R\ast\R[1]$,
there are triangles
$$X_A\s{f_A}\longrightarrow A\s{g_A}\longrightarrow Y_A[1]\s{h_A}\longrightarrow X_A[1],$$
$$X_B\s{f_B}\longrightarrow B\s{g_B}\longrightarrow Y_B[1]\s{h_B}\longrightarrow X_B[1],$$
where $f_A$ (resp. $f_B$) is the minimal right $\X-$approximation of $A$ (resp. $B$)
and $g_A$ (resp. $g_B$) is the minimal left $\Y[1]-$approximation of $A$ (resp. $B$).
Then the composition $u\circ f_A$ factors through $f_B$.  So there exists $s$ such that $f_B\circ s=u\circ f_A$.
\begin{center}
\setlength{\unitlength}{1cm}
\begin{picture}(6,2.3)
\put(-0.15,0){$M$}\put(1.85,0){$A$}
\put(3.85,0){$B$}\put(5.65,0){$M[1]$}
\put(1.85,2){$X_A$}\put(3.85,2){$X_B$}
\mul(0.5,0.1)(2,0){3}{\vector(1,0){1}}
\mul(2,1.8)(2,0){2}{\vector(0,-1){1.4}}
\mul(2.5,2.2)(2,0){1}{\vector(1,0){1}}
\mul(3.5,2.1)(2,0){1}{\vector(-1,0){1}}
\mul(3.9,1.9)(2,0){1}{\vector(-1,-1){1.6}}
\put(2.9,2.3){\small$s$}
\put(2.9,1.9){\small$r$}
\put(1.65,1){\small$f_A$}
\put(3.65,1){\small$f_B$}
\put(2.9,0.2){\small$u$}
\put(4.9,0.2){\small$v$}
\end{picture}
\end{center}
Due to $\Hom(X_B,M[1])=0$, we have $v\circ f_B=0$, then $f_B$ factors through $u$.
Since any morphism from $X_B$ to $A$ factors through $f_A$, then there is a
morphism $r\in\Hom(X_B,X_A)$ such that $f_B=u\circ f_A\circ r$.
Replace $u\circ f_A$ by $f_B\circ s$, we have $f_B=f_B\circ s\circ r$.
Then $s\circ r$ is an isomorphism by the right minimality of $f_B$.
Thus $s$ is a retraction and we have the triangle
$X_A\s{s}\longrightarrow X_B\s{0}\longrightarrow X_C[1]\longrightarrow X_A[1]$, where $X_C$ is a direct summand of $X_A$.
From $f_B\circ s$ and $u\circ f_A$ respectively, by the octahedral axiom, we have the following two commutative diagrams of triangles:
$$\begin{array}{ccccrccr}
 & & X_A & = & X_A & & \\
 & & s\downarrow &&f_B\circ s\downarrow & & \\
Y_B & \longrightarrow & X_B & \s{f_B}\longrightarrow & B & \s{g_B}\longrightarrow & Y_B[1]\\
\parallel & & 0\downarrow & & \downarrow & & \parallel & (\ast)\\
Y_B & \longrightarrow & X_C[1] & \longrightarrow & N & \longrightarrow & Y_B[1]\\
 & & \downarrow & & \downarrow & & \\
  & & X_A[1] & = & X_A[1] & &
\end{array}$$
and
$$\begin{array}{cccccccr}
 & & M & = & M & & \\
 & & w\downarrow & & \downarrow & & \\
X_A & \s{f_A}\longrightarrow & A & \s{g_A}\longrightarrow & Y_A[1] & \s{h_A}\longrightarrow & X_A[1]\\
\parallel & & u\downarrow & & \downarrow & & \parallel & (\ast\ast)\\
X_A & \s{u\circ f_A}\longrightarrow & B & \longrightarrow & N & \longrightarrow & X_A[1]\\
 & & v\downarrow & & \downarrow & & \\
  & & M[1] & = & M[1] .& &
\end{array}$$
 Since the morphism from $Y_B$ to $X_C[1]$ in the third row of the diagram $(\ast)$ is zero,
then $N\cong X_C[1]\oplus Y_B[1]\in\R[1]$. On the other hand, the morphism from $M$ to $Y_A[1]$ in
the third column of the diagram $(\ast\ast)$ is zero due to $\Hom(M,Y_A[1])=0$, then $M[1]$ is isomorphic to a direct summand
of $N$, and then it is in $\R[1]$. Hence $M$ is an object in $\R$.  The functorially finiteness of $\R$ follows from that the
number of indecomposable objects (up to isomorphism) is finite and $\C$ is Hom-finite.
Therefore $\R$ is cluster tilting in $\C$. $\R$ is indecomposable by Corollary \ref{cor1}.

Now we replace $\T$ by $\R$, repeat the proof above. Namely, we consider the following split triangles:
$$X_i\longrightarrow X_i\longrightarrow 0\s{0}\longrightarrow X_i[1],$$
$$0\rightarrow Y_i\rightarrow Y_i[-1][1]\s{0}\rightarrow 0.$$
In these triangles, $X_i\in\X$, $Y_i[-1]\in\Y$. We have that the subcategory $\R'$ generated
by $X_i$, $Y_i[-1]$, $i\in I$ is a cluster tilting subcategory. It is an indecomposable by Corollary \ref{cor1}.
For any $i,j\in I$, $Y_j[-2]\in\Y$, then $\Hom(X_i,Y_j[-1])=\Ext^1(X_i,Y_j[-2])=0$. Note that $\Hom(Y_j[-1],X_i)=\Ext^1(Y_j,X_i)\cong D\Ext^1(X_i,Y_j)=0$. Therefore $X_i\cong0$ for all $i$ or $Y_i\cong0$ for all $i$ as $\R'$ is indecomposable.
Then $\R'\subseteq\Y$ or $\R'\subseteq\X$. Hence $\C=\Y$ or $\C=\X$.

\end{proof}

\begin{rem} The result is not true for $2-$CY triangulated categories without cluster tilting objects. The derived category of coherent sheaves on an algebraic K3 surface is 2-CY and admits no cluster tilting objects. It admits a non trivial t-structure (the canonical t-structure whose heart is the category of coherent sheaves). There are also examples that there are nontrivial t-structures in a 2-CY triangulated category admitting cluster tilting subcategories which contains infinitely many indecomposables (up to isomorphism). For example, the cluster category $\C_{A_{\infty}}$ of type $A_{\infty}$ introduced by Holm-J\"orgensen \cite{HJ,KR07} has non-trivial t-structures (see Theorem 4.1 in \cite{Ng10}). This cluster category has cluster-tilting subcategories containing infinitely many indecomposable objects (see \cite{Ng10} for more details).
\end{rem}









\begin{cor}\label{cor3}
Let $\C$ be a 2-CY triangulated category with a cluster tilting object $T$ and let $\C=\oplus_{j\in J}\C_j$ be the
complete decomposition of $\C$. Then the t-structures in $\C$ are of the form $(\oplus_{j\in L}\C_j,\oplus_{j\in J-L}\C)$
where $L$ is a subset of $J$. In particular, each t-structure has
a trivial heart.
\end{cor}

The following theorem is the second main result in this section, which gives a classification of cotorsion pairs in $2-$CY triangulated categories $\C$ with cluster tilting objects. We note that in those $2-$CY triangulated categories $\C$, any rigid subcategory $\I$ contains only finitely many indecomposables (up to isomorphism) \cite{DK}. So we identify $\I$ with the object $I$ obtained as the direct sum of representatives of isoclasses of indecomposables in it. We also note that for any rigid subcategory $I$ in $\C$, the subquotient category $^\bot (I[1])/I$ is again a $2-$CY triangulated category \cite{IY08}.

\begin{thm}\label{classification}
Let $\C$ be a 2-CY triangulated category admitting cluster tilting objects and $I$ a
rigid subcategory of $\C$. Let $^\bot (I[1])/I=\oplus_{j\in J}I_j$ be the complete decomposition of
$^\bot (I[1])/I$. Then all cotorsion pairs with core $I$ are obtained
as preimages under $\pi:{}^\bot (I[1])\rightarrow {}^\bot (I[1])/I$ of the pairs
$(\oplus_{j\in L}I_j, \oplus_{j\in J-L}I_j)$ where $L$ is a subset of $J$. There are
$2^{ns({}^\bot (I[1])/I)}$ cotorsion pairs with core $I$.
\end{thm}

\begin{proof} Thanks to Theorem 4.7 and Theorem 4.9 in \cite{IY08}, $^\bot (I[1])/I$ is a $2-$CY triangulated category with cluster tilting objects. By Theorem 3.5 and Corollary 3.6 in \cite{ZZ2}, a pair $(\X_1,\X_2)$ of subcategories of $\C$ is a cotorsion pair with core $I$ if and only if
$I\subset\X_i\subset{}^\bot (I[1])$, $i=1,2$, and $(\pi(\X_1),\pi(\X_2))$ is a t-structure in $^\bot (I[1])/I$.
Then by Corollary \ref{cor3}, the t-structures in $^\bot (I[1])/I$ are of the form $(\oplus_{j\in L}I_j, \oplus_{j\in J-L}I_j)$. Therefore the cotorsion pairs with core $I$ are the preimages under $\pi:{}^\bot (I[1])\rightarrow {}^\bot (I[1])/I$ of the $t-$structure $(\oplus_{j\in L}I_j, \oplus_{j\in J-L}I_j)$ in $^\bot (I[1])/I$.
\end{proof}

Indeed, this correspondence is the same as that in Theorem II.2.5 in \cite{BIRS} under the following result: every cotorsion pair is symmetric, i.e.

\begin{cor}\label{cor4}
Let $\C$ be a 2-CY triangulated category admitting a cluster tilting object and let $(\X,\Y)$ be a cotorsion pair with core $I$.
Then $(\Y,\X)$ is also a cotorsion pair with the same core.
\end{cor}

\begin{proof}
By Theorem \ref{classification}, $(\X,\Y)=(\pi^{-1}(\oplus_{j\in L}I_j),\pi^{-1}(\oplus_{j\in J-L}I_j))$ for some subset $J$, then $(\Y,\X)=(\pi^{-1}(\oplus_{j\in J-L}I_j),\pi^{-1}(\oplus_{j\in L}I_j))$ is also a cotorsion pair with core $I$.
\end{proof}

Recall that $(\X,\Y)$ is a co-t-structure in $\C$, if $\Ext^1(\X,\Y)=0$, $\C=\X\ast\Y[1]$ and $\X[-1]\subset\X$, $\Y[1]\subset\Y$.
Using Theorem \ref{classification} and Corollary \ref{cor3}, one can prove that there are no non-trivial co-t-structures
in an indecomposable 2-CY triangulated category with a cluster tilting object in the similar way as \cite{ZZZ}. Indeed,
if $(\X,\Y)$ is a co-t-structure in $\C$, then $(\X,\Y)$ is a cotorsion pair by the definition of co-t-structure.
By Corollary \ref{cor4}, $(\Y,\X)$ is also a cotorsion pair. Since $\Y$ is closed under [1] and $\X$ is closed under
[-1], $(\Y,\X)$ is a t-structure. Then by Theorem \ref{t-str}, $\X=0$ or $\Y=0$.

In fact, we have the following more general result on t-structures or co-t-structures in a $d-$CY triangulated category, generalizing a recent result in \cite{HJY}.

\begin{prop}
Let $\C$ be an indecomposable $d-$CY triangulated category. If $d\geq1$,
then the co-t-structures in $\C$ are $(\C,0)$ and $(0,\C)$. Dually, if $d\leq-1$,
Then the t-structures in $\C$ are $(\C,0)$ and $(0,\C)$.
\end{prop}

\begin{proof}
We only prove the case of $d\geq 1$. Let $(\X,\Y)$ be a co-t-structure in $\C$. For any object $M\in\X\cap\Y$, we have $\Hom(M,M)\cong D\Ext^1(M,M[d-1])=0$
by $M\in\X$ and $M[d-1]\in\Y$. This implies the core of $(\X,\Y)$ is zero. Then by Lemma 2.3, $(\X,\Y)$ is a t-structure.
Thus $\X$, $\Y$ are triangulated subcategories of $\C$.
For any $X\in\X,Y\in\Y$, we have that $\Hom(X,Y)=\Ext^1(X,Y[-1])=0$ and $\Hom(Y,X)\cong D\Ext^1(X,Y[d-1])=0$
by $Y[-1],Y[d-1]\in\Y$. Due to $\C=\X\ast\Y[1]$, we have $\C=\X\oplus\Y$. Therefore $\X=0$ or $\Y=0$.
\end{proof}

\section{Mutations}

In this section, all cluster tilting objects we considered are basic. We shall discuss the relation between mutation of cotorsion pairs and that of cluster tilting objects contained in those cotorsion pairs in a
$2-$CY triangulated category with cluster tilting object. First we
 introduce a notion of cluster tilting subcategories in a subcategory.

\begin{defn} Let $\X$ be a contravariantly finite (or covariantly finite) extension-closed subcategory of a triangulated category $\C$ and
let $\D$ be a subcategory of $\X$. We call that $\D$ is a $\X-$cluster tilting subcategory provided that $\D$ is functorially finite in $\X$,
and satisfies that for any object $D\in \X$, $M\in \D$ if and only if $Ext^1(\D, M)=0$ if and only if $Ext^1(M,\D)=0$.
An object $D$ in $\X$ is called a $\X-$cluster tilting object if $\add D$ is a $\X-$cluster tilting subcategory.

\end{defn}

When $\X=\C$, then $\C-$cluster tilting subcategories are exactly cluster tilting in $\C$.   When $\X$ is a contravariantly finite (or covariantly finite) rigid subcategory, then $\X$ is the only $\X-$cluster tilting subcategory.

From now on to the rest of the section, $\C$ denotes a $2-$CY triangulated category with a cluster tilting object, $(\X,\Y)$ denotes a cotorsion pair with core $I$ in $\C$. We shall show that any cluster tilting object
 containing $I$ as a direct summand in $\C$ gives a $\X-$cluster tilting object and a $\Y-$cluster tilting object respectively. First we prove a lemma.

\begin{lem} Let $\C$ and $(\X,\Y)$ be above. Let $T$ be a cluster tilting object in $\C$. Suppose $T$ can be written as $T=T_{\X}\oplus I\oplus T_{\Y}$ with $T_{\X}\in \X$ and $T_{\Y}\in \Y$.
Then $T_{\X}\oplus I$ is $\X-$cluster tilting and $\T_{\Y}\oplus I$ is $\Y-$cluster tilting.

\end{lem}

\begin{proof} We prove the assertion for $\X-$cluster tilting, the proof for $\Y-$cluster tilting is similar. Suppose that $\Ext^1(T_{\X}\oplus I,X)=0$ for $X\in \X$,
then $\Ext^1(T,X)=0$ since $\Ext^1(T_{\Y},X)\cong D\Ext^1(X,T_{\Y})=0$, the first isomorphism dues to $2-$CY property and the second one dues to that $(\X,\Y)$ is a co-torsion pair. Hence $X\in \add T$. It follows that $X\in \add (T_{\X}\oplus I)$. Then $T_{\X}\oplus I$ is a $\X-$cluster tilting object.
\end{proof}

The following result gives the precise relation between the cluster tilting objects containing $I$ as a direct summand and the $\X-$cluster tilting objects, $\Y-$cluster tilting objects.

\begin{prop} Let $\C$ be a $2-$CY triangulated category with a cluster tilting object, and  $(\X,\Y)$ be a cotorsion pair in $\C$ with core $I$. Then

1. Any cluster tilting object $T$ containing $I$ as a direct summand can be written uniquely as: $T=T_{\X}\oplus I\oplus T_{\Y}$,
such that $T_{\X}\oplus I$ is $\X-$cluster tilting, and $T_{\Y}\oplus I$ is $\Y-$cluster tilting.

2. Any $\X-$cluster tilting object $M$ (or $\Y-$cluster tilting object) contains $I$ as a direct summand, and can be written as $M=M_{\X}\oplus I$ ( $M=M_{\Y}\oplus I$ resp.).
Furthermore $M_{\X}\oplus I\oplus M_{\Y}$ is a cluster tilting object in $\C$.

3. There is a bijection between the set of cluster tilting objects containing $I$ as a direct summand in $\C$ and the product of the set of $\X-$cluster tilting objects with
the set of $\Y-$cluster tilting
  objects. The bijection is given by $T\mapsto T_{\X}\oplus I\oplus T_{Y}$.
\end{prop}

\begin{proof} 1. Let $T$ be any cluster tilting object containing $I$ as a direct summand, we write $T$ as $T=I\oplus T_0$. Then $T_0\in {}^{\bot}(I[1])$, and by passing from $ {}^{\bot}(I[1])$ to the quotient
triangulated category
${}^{\bot}(I[1])/I$, $\underline{T_0}$ is a cluster tilting object in this quotient category by Theorem 4.9 in \cite{IY08}. From the proof of Theorem \ref{classification},  ${}^{\bot}(I[1])/I=\X/I\oplus \Y/I$ as triangulated categories, then $T_0=T_{\X}\oplus T_{\Y}$, where $T_{\X}\in \X, T_{\Y}\in \Y$ such that
$\underline{T_{\X}}, \underline{T_{\Y}}$ are
cluster tilting objects in $\X/I, \Y/I$ respectively. Therefore $T=T_{\X}\oplus I\oplus T_{\Y}$. By Lemma 5.2,  $T_{\X}\oplus I$, $T_{\Y}\oplus I$ are $\X-$cluster tilting, $\Y-$cluster tilting respectively.
\medskip

2. Let $T_1$ be a $\X-$cluster tilting object. Then by $\Ext^1(T_1,I)=0$, we have that $I\in \add T_1$, i.e. $I$ is a direct summand of $T_1$. Then $T_1=T_{\X}\oplus I$. Similarly, any $\Y-$cluster tilting object $T_2$ can be written as $T_2=T_{\Y}\oplus I$. Now $T_{\X}, T_{\Y}$ are cluster tiltings in $\X/I, \Y/I$ respectively, and
then $T_{\X}\oplus T_{\Y}$ is a cluster tilting object in $ {}^{\bot}(I[1])/I$ since $ {}^{\bot}(I[1])/I=
\X/I\oplus \Y/I$. It follows that $T_{\X}\oplus I\oplus T_{\Y}$ is a cluster tilting object in $\C$.
\medskip

3. It follows from $1$ and $2$.
\end{proof}

We know that one can mutate cluster tilting objects to get new ones. In the following we shall see that the mutation of cluster tilting objects containing $I$ as a direct summand is related
to the mutation
 of cotorsion pairs introduced in \cite{ZZ2}. We
recall the notion of mutation of cotorsion pairs in $2-$CY triangulated categories. This notion is defined in a general triangulated category in \cite{ZZ2}.

Let $\C$ be a $2-$CY triangulated category with a cluster tilting object $T$. We denote by $\delta(M)$ the number of indecomposable direct summands (up to isomorphism) of an object $M$. We assume that
$\delta(T)=n$.

Suppose that $(\X,\Y)$ be a cotorsion pair with core $I$. Then $0\le \delta(I)\leq n$ \cite{DK}. It follows from Lemma 2.4 that $\delta(I)=0$ if and only if $(\X,\Y)$ is a t-structure in $\C$,
while $\delta(I)=n$ if and only if $\X=\Y=\add (I)$
is a cluster tilting in $\C$. In the later case, $I$ is a cluster tilting object in $\C$.

\begin{defn} Let $\C$ be an indecomposable $2-$CY triangulated category with a cluster tilting object $T$, and $\delta(T)=n$.
Assume that $0\le d\le n$ is an integer. A cotorsion pair $(\X,\Y)$ with core $I$ is called a $d-$cotorsion pair if $\delta(I)=d$.

Denote by $CTN_d(\C)$ the set of all $d-$cotorsion pairs.

\end{defn}

From the definition above and Theorem \ref{t-str}, $CTN_0(\C)=\{(\C,0),(0,\C)\}$. $CTN_n(\C)$ consists of cluster tilting objects in $\C$.

Let $D$ be a direct summand of $I$ (maybe zero summand). Denote by $\D=\add D.$

Put:

\begin{center}
$\mu ^{-1}(\X;\D):=(\D\ast \X [1])\cap {}^{\bot}(\D [1])$;

$\mu ^{-1}(\Y;\D):=(\D\ast \Y [1])\cap {}^{\bot}(\D [1])$;

$\mu ^{-1}(I;\D):=(\D\ast I [1])\cap {}^{\bot}(\D [1])$.
\end{center}

The following proposition is proved in \cite{ZZ2}.

\begin{prop} With the assumption above, we have that $(\mu^{-1}(\X;\D),\mu^{-1}(\Y;\D))$ is also a cotorsion pair with the core $\mu^{-1}(I;\D)$ in $\C$. Moreover $ (\mu^{-1}(\X;\D),\mu^{-1}(\Y;\D))\in CTN_d(\C)$
 if and only if  $(\X,\Y)\in CTN_d(\C)$.
 \end{prop}

\begin{defn}

We call the cotorsion pair  $(\mu^{-1}(\X;\D),\mu^{-1}(\Y;\D)) $ is a $\D-$mutation of cotorsion pair $(\X,\Y)$.
 Sometimes denote this cotorsion pair by $(\X', \Y')$, denote its core by $I'$.

\end{defn}

\begin{cor} Let $(\X,\Y)$ be a cotorsion pair with core $I$,
and $(\X',\Y')$ with core $I'$ be the $\D-$mutation of $(\X,\Y)$. Then
$(\X',\Y')=(\X,\Y)$ if and only if $I'=I$.

\end{cor}

\begin{proof} The "only if" part is obviously. We prove the "if"
part. Suppose $I'=I$. Then by Theorem 3.11(2) in \cite{ZZ1},  $\D=I$. It follows that passing to the quotient category ${}^{\bot}(I[1])/I$, $(\underline{\X'},\underline{\Y'})$ is $0-$mutation of the t-structure $(\underline{\X},\underline{\Y})$ in the quotient
triangulated category ${}^{\bot}(I[1])/I$. Then $((\underline{\X'},\underline{\Y'})=(\underline{\X},\underline{\Y})$ in this quotient category,
since $\underline{\X},\underline{\Y}$ are triangulated subcategories of ${}^\bot (I[1])/I$ by the proof of Theorem \ref{classification}.
Hence $(\X',\Y')=(\X,\Y)$ in $\C$.

\end{proof}

This corollary was proved for finite triangulated categories in \cite{ZZ2}.

\medskip

Note that there are many choices for $\D$. Two extreme cases are: when $\D=\{0\}$, then the $\D-$mutation of $(\X,\Y)$ is $(\X[1],\Y[1])$; when $\D=\add I$,
then the $\D-$mutation of $(\X,\Y)$ is $(\X,\Y)$ itself.

When $D$ is the
direct summand of $I$ with $\delta(D)=\delta(I)-1$,  the $\D-$mutation is the usually one, which was defined and studied for cluster tilting objects (subcategories) in \cite{BMRRT06,KR07,IY08},
for rigid objects(subcategories) in \cite{MP11}, for maximal
rigid objects(subcategories) in \cite{ZZ1}. We call the $\D-$mutation with $\delta(D)=\delta(I)-1$ just mutation, for simplicity. Denote this mutation by $\mu_{I_0}$, where $I_0$ is the missing
indecomposable object of $D$ in $I$.

\begin{rem}
 For a cluster tilting object $T$, the mutation $\mu$ is an involution. But the mutation of cotorsion pairs is not an involution in general (compare \cite{MP11}), see the following example.

\end{rem}

\begin{exm} Let $Q: 4\rightarrow 3\rightarrow 2\rightarrow 1$, and $\C=D^b(kQ)/\tau^{-1}[1]$, the cluster category of $Q$, see the  AR-quiver below. Set $\X=\add (P_1\oplus P_2\oplus P_3\oplus S_2), \Y=\add (P_2\oplus P_3\oplus P_4\oplus P_4[1]),$ $I=P_2\oplus P_3$.
Then $(\X,\Y)$ is a cotorsion pair with core $I$. We mutate the cotorsion pair $(\X,\Y)$ at $P_2$ to get a new cotorsion pair $(\X_1,\Y_1)$ with core $I_1$, where $\X_1=\add (E\oplus S_3\oplus P_3\oplus P_1),
\Y_1=\add(S_3\oplus P_3\oplus P_4[1]\oplus P_4), I_1=S_3\oplus P_3.$ Now we continues to mutate $(\X_1,\Y_1)$ at $S_3$. We get another new cotorsion pair $(\X_2,\Y_2)$ with core $I_2$, where $\X_2=\add(P_2\oplus S_2\oplus P_3\oplus E),
\Y_2=\add(S_2\oplus P_3\oplus P_4\oplus P_4[1]), I_2=S_2\oplus P_3$. We conclude that $(\X_2,\Y_2)\not= (\X,\Y)$.

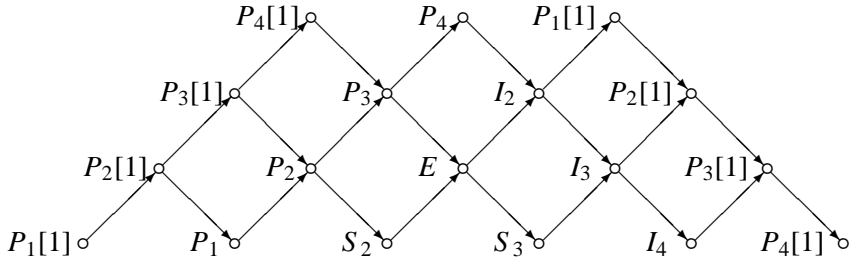
\begin{center}
\setlength{\unitlength}{1cm}
\begin{picture}(10,4)

\mul(0,0)(2,0){6}{$\c$}
\mul(1,1)(2,0){5}{$\c$} \mul(2,2)(2,0){4}{$\c$}
\mul(3,3)(2,0){3}{$\c$}

\mul(0.15,0.15)(2,0){5}{\vector(1,1){0.9}}
\mul(1.15,1.15)(2,0){4}{\vector(1,1){0.9}}
\mul(1.15,1.05)(2,0){5}{\vector(1,-1){0.9}}
\mul(2.15,2.05)(2,0){4}{\vector(1,-1){0.9}}

\mul(2.15,2.15)(2,0){3}{\vector(1,1){0.9}}
\mul(3.15,3.05)(2,0){3}{\vector(1,-1){0.9}}

\put(-0.9,0){${P_1[1]}$} \put(0.1,1){$P_2[1]$}
\put(1.1,2){$P_3[1]$} \put(2.1,3){$P_4[1]$}

\put(1.5,0){$P_1$} \put(2.5,1){$P_2$}
\put(3.5,2){$P_3$} \put(4.5,3){$P_4$}

\put(3.5,0){$S_2$} \put(5.5,0){$S_3$} \put(4.5,1){$E$}

 \put(5.5,2){$I_2$} \put(6.5,1){$I_3$} \put(7.5,0){$I_4$}
\put(6,3){${P_1[1]}$} \put(7,2){$P_2[1]$}
\put(8,1){$P_3[1]$} \put(9,0){$P_4[1]$}
\end{picture}
\end{center}

\end{exm}

We define mutation quiver of cotorsion pairs in $\C$. It is a quiver whose vertices are cotorsion pairs, there is an arrow from the vertex to another vertex if the target
 cotorsion pair is a mutation of the initial one. This quiver is denoted by $\M(\C)$. It is not connected from Proposition 5.5.
 Denoted by $\M_d(\C)$ the subquiver of $\M(\C)$ consisting of vertices belong to
$CTN_d(\C)$. $\M(\C)= \bigsqcup_{d=0}^n\M_d(\C)$. Note that if we replace the each double anti-arrows by an edge, then $\M_n(\C)$ is the exchange graph of cluster tilting objects in $\C$. This graph
is conjectured to be connected for every indecomposable $2-$CY triangulated category \cite{Re10}.

Now we give the relation of mutation of cluster tilting objects containing $I$ as a direct summand with the mutation of cotorsion pairs.

\begin{prop}
 Let $(\X,\Y)$ be a cotorsion pair with core $I$ in $\C$, $T=T_{\X}\oplus I\oplus T_{\Y}$ a cluster tilting object.
Suppose $(\X',\Y')$ is a $\D-$mutation of $(\X,\Y)$, $I'$ is the core of $(\X',\Y')$. Then the $\D-$mutation $T'$ of $T$ is
$T'_{\X'}\oplus I'\oplus T'_{Y'}$.
\end{prop}

\begin{proof} For $\D=\add D$, where $D$ is a direct summand of $I$, we consider the subquotient category ${}^{\bot}(\D [1])/\D$. It is a triangulated category by \cite{IY08} with shift functor $<1>$.

In this subquotient category, $(\underline{\X},\underline{\Y})$ is a cotorsion pair with core $\underline{I}$ in \cite{ZZ2} and $\underline{T}=\underline{T_{\X}}\oplus \underline{I}\oplus \underline{T_{\Y}}$ is a cluster tilting object by \cite{IY08}. The images of their $\D-$mutations are $(\underline{\X'},\underline{\Y'})=(\underline{\X}<1>, \underline{\Y}<1>)$, $\underline{T'}=\underline{T}<1>=\underline{T_{\X}}<1>\oplus \underline{I}<1>\oplus \underline{T_{\Y}}<1>$ respectively. It follows that $\underline{T_{\X}}<1>\in \underline{\X}<1>$, $\underline{T_{\Y}}<1>\in \underline{ \Y}<1>$.
Therefore $T'=T'_{\X'}\oplus I'\oplus T'_{\Y'}$, where $T'_{\X'}\oplus I'$,
$T'_{\Y'}\oplus I' $ are $\X'-$cluster tilting object in $\X'$, $\Y'-$cluster tilting object in $\Y'$ respectively.

\end{proof}

Now we state and prove the main result in this section.

\begin{thm}  Let $(\X,\Y)$ be a cotorsion pair with core $I$ in a $2-$CY triangulated category $\C$ with a cluster tilting object. Let $T=T_{\X}\oplus I\oplus T_{\Y}$ be a cluster tilting object containing $I$
as a direct summand. Suppose that $T_0$ is an indecomposable direct summand of $T$. We consider the mutation $\mu_{T_0}(T)$ of $T$ in $T_0$.

$1.$ If $T_0$ is a direct summand of $I$, then $\mu_{T_0}(T)=T'_{\X'}\oplus I'\oplus T'_{\Y'}$, where $(\X',\Y')=\mu_{T_0}(\X,\Y)$ is the mutation of $(\X, \Y)$, $I'$ is the core of cotorsion pair $(\X',\Y')$.

$2.$ If $T_0$ is not the direct summand of $I$, then $\mu_{T_0}(T)=\mu_{T_0}(T_{\X}\oplus I)\oplus T_{\Y}$ when $T_0$ is a direct summand of $T_{\X}$, and  $\mu_{T_0}(T)=T_{\X} \oplus \mu_{T_0}(I\oplus T_{\Y})$ when
 $T_0$ is a direct summand of $T_{\Y}$.

\end{thm}

\begin{proof} 1. The assertion follows from Proposition 5.9.

2. We will prove the case of that $T_0$ is a direct summand of $T_{\X}$, the proof for the other case is similar.
We first note that any morphism $f: X\rightarrow Y$ with $X\in \X, Y\in \Y$ factors through the core $I$. This dues to the fact the image of $f$ under the
projection $\pi: {}^{\bot}(I[1])\rightarrow {}^{\bot}(I[1])/I$ is zero since ${}^{\bot}(I[1])/I=\X/I\oplus \Y/I$ as triangulated categories. It follows that for the minimal left $T/T_{0}-$approximation of $T_0$, say
$g: T_0\rightarrow B$, we have $B\in \add (T_{\X}\oplus I)$. Then $g: T_0\rightarrow B$ is a minimal left $(T_{\X}\oplus I)-$approximation. Extend $g$ to a triangle $T_0\s{g}{\rightarrow}B\rightarrow T_0'\rightarrow T_0[1]$. It induces a triangle in the subfactor triangulated category ${}^{\bot}(I[1])/I: T_0\s{\underline{g}}{\rightarrow}B\rightarrow T_0'\rightarrow T_0<1>$ [IY]. It follows that $T_0'\in \X/I$ and $T_0'\in \X$. Then $\mu_{T_0}(T)=(T/T_0)\oplus T_0' =T_0'\oplus (T_{\X}/T_0)\oplus I\oplus T_{\Y}=\mu_{T_0}(T_{\X}\oplus I)\oplus T_{\Y}$.

\end{proof}

\begin{rem} For any cotorsion pair $(\X,\Y)$ with core $I$ in $\C$.  From the theorem above, $\X$ (or $\Y$) has weak cluster structure in the sense of \cite{BIRS}, i.e. the $\X-$cluster tiltings $T_{\X}\oplus I$ are the candidates
 of extended
clusters, where $I$ is the set of coefficients; one can mutate the $\X-$cluster tiltings at $T_0$ to get a new $\X-$cluster tilting by the above theorem; and one also have exchange triangles. There is a
substructure of $\C$ induced by a $\X-$cluster tilting and a $\Y-$cluster tilting: Let $T_{\X}\oplus I$ be a $\X-$cluster tilting, $T_{\Y}\oplus I$ a $\Y-$cluster tilting. Then $T_{\X}\oplus I\oplus T_{\Y}$ is the
cluster tilting in $\C$ by Proposition 5.3. We call that  $T_{\X}\oplus I$ and $T_{\Y}\oplus I$ give a substructure of $\C$ (compare \cite{BIRS}) if for any $\X-$cluster tilting $T'_{\X}\oplus I$,
 $\Y-$cluster tilting $T'_{\Y}\oplus I$,  both of which are obtained from $T_{\X}\oplus I$ and $T_{\Y}\oplus I$ respectively via a finite number of mutations, the cluster tilting object $T'_{\X}\oplus I\oplus T'_{\Y}$ can be obtained from
$T_{\X}\oplus I\oplus T_{\Y}$ via a finite number of mutations in $\C$.

\end{rem}

\section{Hearts of cotorsion pairs}

As an application of the classification theorem of cotorsion pairs, we determine
the hearts of cotorsion pairs in $2-$CY triangulated categories with cluster tilting objects in this section. Hearts of cotorsion pairs in any triangulated category were introduced by Nakaoka \cite{Na11}, which unify
the construction of hearts of t-structures \cite{BBD81} and construction of the abelian quotient categories
by cluster tilting subcategories \cite{BMRRT06,KR07,KZ}.

We recall the construction of hearts of cotorsion pairs from \cite{Na11}: Let $\C$ be a triangulated category and $(\X,\Y)$ a cotorsion pair with core $\I$ in $\C$. Denote by $\H$ the subcategory
$(\X[-1]\ast \I)\cap (\I\ast\Y[1])$. The heart of the cotorsion pair $(\X,\Y)$ is
defined as the quotient category $\H/\I$, denoted by $\underline{\H}$.
\medskip

It was proved that $\underline{\H}$ is an abelian category \cite{Na11}. There is a cohomology functor $H=h \pi$ from $\C$ to $\underline{\H}$, where $\pi$ is the quotient functor from $\C$ to $\underline{\C}=\C/\I$ and
$h$ is a functor from $\underline{\C}$ to $\underline{\H}$. Those constructions were given in Proposition 3.4 and Proposition 4.2 in \cite{AN} combined with Construction 4.2, Proposition 4.3 and Remark 4.5 in \cite{Na11}. For the convenience of reader, we recall the definitions of the functor $h$ from \cite{AN} as follows.

For any $M\in\C$, there is a triangle
$Y_M\rightarrow X_M\rightarrow M\rightarrow Y_M[1]$ with $X_M\in\X,Y_M\in\Y$, since $(\X,\Y)$ is a cotorsion pair. Then there is a triangle $X_M'[-1]\rightarrow X_M\rightarrow Y_M'\rightarrow X_M'$ with $X_M'\in\X,Y_M'\in\Y$, since $(\X[-1],\Y[-1])$ is a cotorsion pair. Composing the morphism from $X_M'[-1]$ to $X_M$ and the morphism from $X_M$ to $M$, we have the following commutative diagram of triangles in $\C$ by the octahedral axiom, in which we get $\widetilde{M}$ and $s_M:M\rightarrow \widetilde{M}$:
$$\begin{array}{ccccccc}
&&X_M'[-1]&=&X_M'[-1]&&\\
&&\downarrow&&\downarrow&&\\
Y_M&\longrightarrow&X_M&\longrightarrow&M&\longrightarrow&Y_M[1]\\
\parallel&&\downarrow&&s_M\downarrow\ \ \  \ \ \ &&\parallel\\
Y_M&\longrightarrow&Y_M'&\longrightarrow&\widetilde{M}&\longrightarrow&Y_M[1]\\
&&\downarrow&&\downarrow&&\\
&&X_M'&=&X_M'&&
\end{array}(\star).$$
Using the definition of cotorsion pair $(\X[-1],\Y[-1])$ again, we have a triangle $X_M''[-1]\rightarrow \widetilde{M}\rightarrow Y_M''\rightarrow X_M''$ and then we have another triangle $Y_M'''\rightarrow X_M'''\rightarrow Y_M''\rightarrow Y_M'''[1]$ with $X_M'',X_M'''\in\X$ and $Y_M'',Y_M'''\in\Y$. Compose the morphism from $X_M'''$ to $Y_M''$ and the morphism from $Y_M''$ to $X_M''$, by the octahedral axiom, we have the following commutative diagram of triangles in $\C$, in which we have $\overline{M}$ and $t_M:\overline{M}\rightarrow \widetilde{M}$:
$$\begin{array}{ccccccc}
&&Y_M'''&=&Y_M'''&&\\
&&\downarrow&&\downarrow&&\\
X_M''[-1]&\rightarrow&\overline{M}&\rightarrow&X_M'''&\rightarrow&X_M''\\
\parallel&&t_M\downarrow\ \ \ \ \ &&\downarrow&&\parallel\\
X_M''[-1]&\rightarrow&\widetilde{M}&\rightarrow&Y_M''&\rightarrow&X_M''\\
&&\downarrow&&\downarrow&&\\
&&Y_M'''[1]&=&Y_M''''[1]&&
\end{array}(\star\star).$$
The image of $M$ under $h$ is defined as $\overline{M}$. Abe and Nakaoka proved that $\overline{M}\in\underline{\H}$.
It is easy to see that up to isomorphisms in $\underline{\H}$, $\overline{M}$ does not depend on the choice of $X_M,X_M',X_M'',X_M'''$
and $Y_M,Y_M',Y_M'',Y_M'''$ (See Section 4 in \cite{AN} for details).

For any morphism $f:M\rightarrow N$ in $\underline{\C}$,
there is a unique morphism $\widetilde{f}$ in $\underline{\C}$ such that the left square of the following diagram commutate (Proposition 4.3 in \cite{Na11}) and then there is a unique morphism $\overline{f}$ in $\underline{\C}$ such that the right square in the following diagram commutate (Remark 4.5 in \cite{Na11}):
$$\begin{array}{ccccc}
M&\s{\underline{s_M}}\rightarrow& \widetilde{M}&\s{\underline{t_M}}\leftarrow&\overline{M}\\
f\downarrow&&\widetilde{f}\downarrow\ \ \ \ &&\overline{f}\downarrow\ \ \ \ \\
N&\s{\underline{s_N}}\rightarrow& \widetilde{N}&\s{\underline{t_N}}\leftarrow&\overline{N}
\end{array}(\star\star\star).$$
The image of $f$ under $h$ is defined as $\overline{f}$.

\medskip

We state two simple facts followed from the constructions above.

\begin{lem}
$H(\X)=0$ and $H(\Y)=0$ hold.
\end{lem}

\begin{proof} We give a proof for $H(\X)=0$, $H(\Y)=0$ can be proved dually.
Let $M$ be an object in $\X$. One can choose $Y_M=0$. Then $\widetilde{M}\cong Y_M'$.
So one can choose $X_M''=0$. Then $h(M)=\overline{M}\cong X_M'''$. Note that $X_M'''\in\Y\ast\Y\subset\Y$ and $\I=\X\cap\Y$. We have that $h(M)\in \I$ and hence $h(M)\cong 0$ in $\underline{\H}$.
\end{proof}

\begin{lem}
$h|_{\underline{\H}}=\id_{\underline{\H}}$.
\end{lem}

\begin{proof}
By the definition of $h$, one only need to check that $h(M)\cong M$ for any $M\in\underline{\H}$.
In this case, we have that $X_M\in \I$ by Corollary 3.3 in \cite{Na11}. One can choose $Y_M'=X_M$ and then
$\widetilde{M}\cong M$. By the dual, one can have that $\overline{M}\cong\widetilde{M}$. Thus this lemma holds.
\end{proof}

Let $(\X_1,\Y_1)$ and $(\X_2,\Y_2)$ be two cotorsion pairs with the same core $\I$ in a triangulated category $\C$.
Denote by $\underline{H_i}$ the heart of $(\X_i,\Y_i)$, $i=1,2$.
Let $H_i=h_i\pi$
be the cohomology functor from $\C$ to $\underline{\H_i}$ given in \cite{AN},
and $\iota_i$ be the inclusion functor from $\underline{\H_i}$ to $\underline{\C}$, $i=1,2$.
The composition functors $h_1\iota_2$ and $h_2\iota_1$ are denoted by $E$ and $F$ respectively.

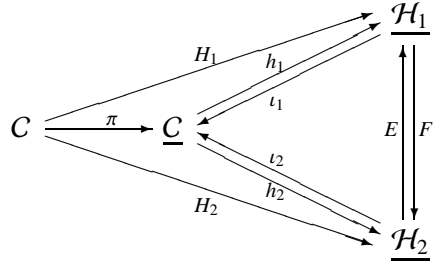
\begin{center}
\setlength{\unitlength}{1cm}
\begin{picture}(5.5,4)

\put(-0.15,2){$\C$}
\put(0.3,2.1){\vector(1,0){1.4}}\put(1.1,2.15){\scriptsize$\pi$}
\put(1.85,2){$\underline{\C}$}
\put(2.3,2.3){\vector(2,1){2.4}}\put(3.2,2.9){\scriptsize$h_1$}
\put(4.85,3.5){$\underline{\H_1}$}
\put(2.3,1.9){\vector(2,-1){2.4}}\put(3.2,1.2){\scriptsize$h_2$}
\put(4.85,0.5){$\underline{\H_2}$}
\put(4.7,3.35){\vector(-2,-1){2.4}}\put(3.25,2.4){\scriptsize$\iota_1$}
\put(4.7,0.85){\vector(-2,1){2.4}}\put(3.25,1.65){\scriptsize$\iota_2$}
\put(5.15,3.2){\vector(0,-1){2.3}}\put(4.75,2){\scriptsize$E$}
\put(5,0.9){\vector(0,1){2.3}}\put(5.2,2){\scriptsize$F$}
\put(0.3,2.2){\vector(3,1){4.3}}\put(2.25,3){\scriptsize$H_1$}
\put(0.3,2){\vector(3,-1){4.3}}\put(2.25,1){\scriptsize$H_2$}
\end{picture}
\end{center}

\begin{lem}
If $H_1({}^\bot (\I[1]))=0$ and $H_1((\I[-1])^\bot)=0$,
then $EF\simeq\id_{\underline{\H_1}}$.
\end{lem}

\begin{proof}
For any $M\in\underline{\H_1}$, we have the above commutative diagrams $(\star)$ and $(\star\star)$
with $X_M$, $X_M'$, $X_M''$, $X_M'''\in\X_2$ and $Y_M,Y_M',Y_M'',Y_M'''\in\Y_2$.
Then $h_2(M)=\overline{M}$ by the definition.
The first and the last morphisms in the third column of
the diagram $(\star)$ and in the second column of the diagram $(\star\star)$ factor
through ${}^\bot (\I[1])$ or $(\I[-1])^\bot$ respectively, by $\X_2\subset{}^\bot (I[1])$
and $\Y_2\subset (I[-1])^\bot$. Then the image of these morphisms under
$H_1$ are zero. Applying the cohomology functor $H_1$ to these two triangles
(in the third column of the diagram $(\star)$ and in the second column of the
 diagram $(\star\star)$), one has two isomorphisms in $\underline{\H_1}$:
$$H_1M\s{H_1(s_M)}\longrightarrow E\widetilde{M}$$
and
$$EFM\s{H_1(t_{M})}\longrightarrow E\widetilde{M}.$$
Since $M\in\underline{\H_1}$, so $H_1M=M$ by Lemma 6.2.

For any morphism $f:M\rightarrow N$ in $\underline{\H_1}$, applying the functor $h_1$ to the diagram $(\star\star\star)$, we have the following
commutative diagram in $\underline{\H_1}$:
$$\begin{array}{ccc}
h_1M&\s{H_1(t_{M})^{-1}H_1(s_M)}\longrightarrow&EFM\\
h_1f\downarrow\ \ \ \ \ &&EFf\downarrow\ \ \ \ \ \ \ \\
h_1N&\s{H_1(t_{N})^{-1}H_1(s_N)}\longrightarrow&EFN
\end{array}.$$
Since $M,N\in\underline{H_1}$, then $h_1M=M$, $h_1N=N$ and $h_1f=f$ by Lemma 6.2. Therefore, $\id_{\underline{\H_1}}\simeq EF$.

\end{proof}

From now on to the end of this section, we assume that $\C$ is a 2-CY triangulated category. We continue to use the same notations as above. Fixed a cotorsion pair $(\X,\Y)$ with core $\I$, which is assumed functorially finite in $\C$ (e.g. $\I$ contains only finitely many indecomposable objects). Let $(\X_1,\Y_1)$ be the cotorsion pair $(\I,{}^\bot (\I[1]))$ and $(\X_2,\Y_2)=(\X, \Y)$. Then the condition of Lemma 6.3 holds automatically by Lemma 6.1 and the heart $\H_1$ is equivalent to the module category over $\I$ \cite{IY08}, denoted by $\mod\ \I$. By Corollary 3.6 in \cite{ZZ2}, we have that $(\X/\I,\Y/\I)$ is a t-structure in the 2-CY triangulated category ${}^\bot (\I[1])/\I$. Recall that the shift functor $\langle1\rangle$ in the triangulated category ${}^\bot (\I[1])/\I$ defined in \cite{IY08} is obtained by the following triangle: $M\rightarrow I_M\rightarrow M\langle1\rangle\rightarrow M[1]$, where $M\in{}^\bot( \I[1])$, $I_M\in \I$. We denote the heart $(\X/\I)\langle-1\rangle\bigcap(\Y/\I)\langle1\rangle$ of this t-structure by $\A$ which is an abelian category \cite{BBD81}.

\begin{lem}
The category $\A$ is an abelian subcategory of the heart $\underline{\H}$ of $(\X,\Y)$.
\end{lem}

\begin{proof} Obviously $\A, \underline{\H}$ are the subcategories of $\underline{\C}$.
Since $\X\langle-1\rangle\subset\X[-1]*\I$ and $\Y\langle1\rangle\subset \I*\Y[1]$, then we have that $$\A=(\X/\I)\langle-1\rangle\bigcap(\Y/\I)\langle1\rangle=\X\langle-1\rangle/\I\bigcap\Y\langle1\rangle/\I
\subset\underline{\H}.$$
\end{proof}





 For the subcategory $\A$ of $\underline{\H}$, we use $\underline{\H}/\A$ to denote the quotient category, whose objects are the same as $\underline{\H}$, whose morphisms are the factor additive group $\Hom_{\underline{\H}/\A}(X,Y)=\Hom_{\underline{\H}}(X,Y)/\A(X,Y)$, for $X,Y\in \underline{\H}$. Where $\A(X,Y)$ is the subgroup of $\Hom_{\underline{\H}}(X,Y)$ consisting of morphisms which factor through an object in $\A$. It is an additive category, and  the natural projection $\pi_{\A}:\underline{\H}\rightarrow \underline{\H}/\A$ is an additive functor.
 \bigskip

 Since $H_1(\A)=0$ by $\A\subset{}^\bot (\I[1])/\I$, we have $E(\A)=0$. Then $E$ induces an additive functor $E':\underline{\H}/\A\rightarrow \mod \ \I$ which makes the following diagram commute:

\begin{center}
\setlength{\unitlength}{1cm}
\begin{picture}(5,2.5)
\put(1.8,0.1){$\underline{\H}/\A$}
\put(0.1,2.1){$\underline{\H}$}
\put(4.1,2.1){$\mod\ \I$}
\put(0.5,2){\vector(1,-1){1.6}}
\put(0.5,2.1){\vector(1,0){3.5}}
\put(4,2.3){\vector(-1,0){3.5}}
\put(2.5,0.4){\vector(1,1){1.6}}
\put(2.1,2.4){$F$}
\put(2.1,1.8){$E$}
\put(0.75,0.95){$\pi_\A$}
\put(3.4,0.95){$E'$}
\end{picture}
\end{center}

We have that $E'\pi_\A F= EF\simeq \id_{\mod\ I}$ by Lemma 6.3. On the other hand, we have that $\pi_\A F E'\pi_\A=\pi_\A F E\simeq\pi_\A$ which implies $\pi_\A F E'\simeq\id_{\underline{\H}/\A}$. Thus we have the main result of this section which determines hearts of any cotorsion pairs in $2-$CY triangulated categories with cluster tilting objects.

\begin{thm}
Let $\C$ be a 2-CY triangulated category and $(\X,\Y)$ be a cotorsion pair in $\C$ with core $\I$. Assume that $\I$ is functorially finite. Then we have an equivalence of additive categories
$$\underline{\H}/\A\simeq\mod\ \I,$$
where $\underline{\H}$ is the heart of $(\X,\Y)$, $\A$ is the heart of $(\X/\I,\Y/\I)$. If $\C$ has cluster tilting objects, then we have an equivalence of abelian categories
$$\underline{\H}\simeq\mod\ \I,$$
and in particular, the hearts of any two cotorsion pairs with the same core are equivalent.
\end{thm}

\begin{proof} From the above, we have the functor $E':\underline{\H}/\A\rightarrow mod \I$, and the functor $\pi_{\A}F: mod \I\rightarrow \underline{\H}/\A$. Those functors satisfy $\pi_{\A}F E'\cong id_{mod\I}$ and $E'\pi_{\A}F\cong id_{\underline{\H}/\A}.$ Then $$\underline{\H}/\A\simeq \mod\ \I.$$
We prove the second assertion. If $\C$ has cluster tilting object, then the core of every cotorsion pair is functorially finite, since the core contains only finite non-isomorphic indecomposable objects. By Corollary \ref{cor3}, $\A$ is trivial. Then we have the equivalence $\underline{\H}\simeq\mod\ \I$. Note that both categories $\underline{\H}$ and $\mod\ \I$ are abelian and $E, F$ are additive functors. So $\underline{\H}$ and $\mod\ \I$ are isomorphic as abelian categories.
\end{proof}

\begin{exm}
Let $Q: 4\rightarrow 3\rightarrow 2\rightarrow 1$, $\C$ the cluster category of $Q$.
Set $I=\add(P_2[1]\oplus P_3[1])$. Then the subcategory
${}^\bot (I[1])=\add(P_1[1]\oplus P_2[1]\oplus P_3[1]\oplus P_4[1]\oplus I_4\oplus P_1)$.
We mark the indecomposable objects in ${}^\bot (I[1])$ by $\Box$ in the following AR-quiver of $\C$.

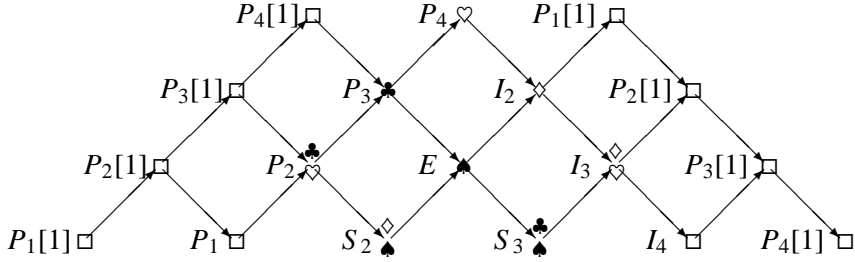
\begin{center}
\setlength{\unitlength}{1cm}
\begin{picture}(10,4)

\mul(0,0)(1,1){4}{$\Box$}
\mul(7,3)(1,-1){4}{$\Box$}
\put(2,0){$\Box$} \put(8,0){$\Box$}

\mul(0.15,0.15)(2,0){5}{\vector(1,1){0.9}}
\mul(1.15,1.15)(2,0){4}{\vector(1,1){0.9}}
\mul(1.15,1.05)(2,0){5}{\vector(1,-1){0.9}}
\mul(2.15,2.05)(2,0){4}{\vector(1,-1){0.9}}

\mul(2.15,2.15)(2,0){3}{\vector(1,1){0.9}}
\mul(3.15,3.05)(2,0){3}{\vector(1,-1){0.9}}

\put(-0.9,0){${P_1[1]}$} \put(0.1,1){$P_2[1]$}
\put(1.1,2){$P_3[1]$} \put(2.1,3){$P_4[1]$}

\put(1.5,0){$P_1$} \put(2.5,1){$P_2$}
\put(3.5,2){$P_3$} \put(4.5,3){$P_4$}

\put(3.5,0){$S_2$} \put(5.5,0){$S_3$} \put(4.5,1){$E$}

 \put(5.5,2){$I_2$} \put(6.5,1){$I_3$} \put(7.5,0){$I_4$}
\put(6,3){${P_1[1]}$} \put(7,2){$P_2[1]$}
\put(8,1){$P_3[1]$} \put(9,0){$P_4[1]$}

\put(3,1.2){$\clubsuit$}\put(4,2){$\clubsuit$}\put(6,0.2){$\clubsuit$}
\put(4,0.2){$\diamondsuit$}\put(6,2){$\diamondsuit$}\put(7,1.2){$\diamondsuit$}
\put(3,0.9){$\heartsuit$}\put(5,3){$\heartsuit$}\put(7,0.9){$\heartsuit$}
\put(4,-0.1){$\spadesuit$}\put(5,1){$\spadesuit$}\put(6,-0.1){$\spadesuit$}

\end{picture}
\end{center}
There are four cotorsion pairs with core $I$ in this category, we list them together with
their hearts in the following and mark the indecomposable objects in each heart by
$\clubsuit$, $\diamondsuit$, $\heartsuit$ and $\spadesuit$ respectively in order in the AR-quiver above.

$\begin{array}{cc}
\text{Cotorsion pairs} & \text{Hearts} \\
(I,{}^\bot I[1]) & \add(P_2\oplus P_3\oplus S_3)\\
({}^\bot I[1],I) & \add(S_2\oplus I_2\oplus I_3)\\
(\add(P_2[1]\oplus P_3[1]\oplus P_4[1]\oplus I_4),\add(P_2[1]\oplus P_3[1])\oplus P_1[1]\oplus P[1])) & \add(P_2\oplus P_4\oplus I_3)\\
(\add(P_2[1]\oplus P_3[1])\oplus P_1[1]\oplus P[1]),\add(P_2[1]\oplus P_3[1]\oplus P_4[1]\oplus I_4)) & \add(S_2\oplus E\oplus S_3)
\end{array}$

\end{exm}

\begin{center}
\textbf {ACKNOWLEDGMENTS.}\end{center}

This work was completed when the second author was visiting at Universit\"at Bielefeld supported by DAAD.
He would like to thank Claus Michael Ringel and Henning Krause for suggestions and for hospitality, and to thank DAAD for financial support. Both authors would like to
thank Idun Reiten and Bernhard Keller for discussion.

\end{document}